\documentclass[letterpaper]{amsart}
\usepackage{amsfonts,amssymb,amscd,amsmath,enumerate,url,verbatim}
\usepackage[all]{xy}
\SelectTips{eu}{}
\xyoption{curve}
\CompileMatrices
\usepackage[colorlinks]{hyperref}

\newcommand{\lra}{\longrightarrow}
\newcommand{\xra}{\xrightarrow}
\newcommand{\xla}{\xleftarrow}

\newcommand{\ov}{\overline}

\newcommand{\ges}{{\scriptscriptstyle\geqslant}}
\newcommand{\col}{\colon}
\newcommand{\wh}{\widehat}
\newcommand{\ul}{\underline}

\newcommand{\dd}{\partial}
\newcommand{\fm}{{\mathfrak m}}
\newcommand{\fn}{{\mathfrak n}}
\newcommand{\fq}{{\mathfrak q}}
\newcommand{\fp}{{\mathfrak p}}
\newcommand{\vf}{{\varphi}}
\newcommand{\bd}{\boldsymbol}
\newcommand{\bw}{{\mathsf\Lambda}}
\newcommand{\Sym}{{\mathsf S}}
\newcommand{\Gam}{{\mathsf\Gamma}}

\newcommand{\card}{\operatorname{card}}
\newcommand{\id}{\operatorname{id}}
\newcommand{\edim}{\operatorname{edim}}
\newcommand{\Ker}{\operatorname{Ker}}
\newcommand{\depth}{\operatorname{depth}}
\newcommand{\grade}{\operatorname{grade}}
\newcommand{\height}{\operatorname{height}}

\newcommand{\Ass}{\operatorname{Ass}}

\newcommand{\Ann}{\operatorname{Ann}}
\newcommand{\Spec}{\operatorname{Spec}}

\newcommand{\Supp}{\operatorname{Supp}}
\newcommand{\cidim}{\operatorname{CI-dim}}

\newcommand{\pd}{\operatorname{pd}}
\newcommand{\fd}{\operatorname{fd}}
\newcommand{\cmd}{\operatorname{cmd}}
\newcommand{\cha}[1]{\operatorname{char}(#1)}
\newcommand{\cid}{\operatorname{cid}}

\newcommand{\rank}{\operatorname{rank}}
\newcommand{\HH}[2]{\operatorname{H}_{#1}(#2)}

\newcommand{\Tor}[4]{\operatorname{Tor}_{#1}^{#2}(#3,#4)}
\newcommand{\Ext}[4]{\operatorname{Ext}^{#1}_{#2}(#3,#4)}
\newcommand{\AQ}[4]{\operatorname{D}_{#1}(#2|#3,#4)}
\newcommand{\Hom}[3]{\operatorname{Hom}_{#1}(#2,#3)}
\newcommand{\Poi}[3]{{P}_{#1}^{#2}(#3)}

\newcommand{\Hil}[2]{{H}_{#1}(#2)}

\newcommand{\cls}{\operatorname{cls}}
\newcommand{\shift}{{\mathsf\Sigma}}
\newcommand{\EC}[2]{\operatorname{E}_{#1}^{#2}}
\newcommand{\EE}[2]{\operatorname{E}^{#1}_{#2}}
\newcommand{\Ed}[2]{\operatorname{d}^{#1}_{#2}}

\theoremstyle{plain}
\newtheorem{theorem}{Theorem}[section]

\newtheorem{proposition}[theorem]{Proposition}
\newtheorem{lemma}[theorem]{Lemma}
\newtheorem{corollary}[theorem]{Corollary}

{\Alph{theorem}}

\newtheorem{bchunk}[theorem]{}

\theoremstyle{definition}

\newtheorem{construction}[theorem]{Construction}
\newtheorem{chunk}[theorem]{}

\theoremstyle{remark}
\newtheorem{example}[theorem]{Example}

\newtheorem{question}[theorem]{Question}

\newtheorem*{Claim}{Claim}

\newtheorem{conjecture}[theorem]{Conjecture}

\newtheorem*{Remark}{Remark}
\newtheorem{remark}[theorem]{Remark}

\numberwithin{equation}{theorem}

\begin{document}
\title[Quasi-complete intersection homomorphisms]
{Quasi-complete intersection homomorphisms}
\date{\today}

\author[L.~L.~Avramov]{Luchezar L.~Avramov}
\address{Luchezar L.~Avramov\\ Department of Mathematics\\
   University of Nebraska\\ Lincoln\\ NE 68588\\ U.S.A.}
     \email{avramov@math.unl.edu}

\author[I.~B.~Henriques]{In\^es Bonacho Dos Anjos Henriques}
\address{In\^es Bonacho Dos Anjos Henriques\\ Department of Mathematics\\
   University of Nebraska\\ Lincoln\\ NE 68588\\ U.S.A.}
\curraddr{Department of Pure Mathematics\\ University of Sheffield\\ Hicks Building\\
S3 7RH\\ U.K.}
    \email{i.henriques@sheffield.ac.uk}

\author[L.~M.~\c{S}ega]{Liana M.~\c{S}ega}
\address{Liana M.~\c{S}ega\\ Department of Mathematics and Statistics\\
   University of Missouri\\ \linebreak Kansas City\\ MO 64110\\ U.S.A.}
     \email{segal@umkc.edu}
     
\dedicatory{This paper is dedicated to the memory of Andrey Todorov}


\thanks{Research partly supported by NSF grants DMS-0803082 and DMS-1103176
(LLA), NSF grant DMS-1101131 and Simons Foundation grant 20903 (LS)}

 \begin{abstract}
Extending a notion defined for surjective maps by Blanco, Majadas and Rodicio, 
we introduce and study a class of homomorphisms  of commutative noetherian rings, 
which strictly contains the class of locally complete intersection homomorphisms 
while sharing many of its remarkable properties.
 \end{abstract}

  \maketitle

\section*{Introduction}

This paper is concerned with those homomorphisms $\vf\col R\to S$ of commutative noetherian 
rings for which the homology functors $\AQ nSR-$ of Andr\'e \cite{AnBook} and Quillen \cite{Qu}
vanish for $n\ge3$.  We call them quasi-complete intersection (or q.c.i.)\ homomorphisms, in view 
of the characterization of locally complete intersection (or l.c.i.)\ homomorphisms by the condition 
$\AQ nSR-=0$ for $n\ge2$; see \cite{Qu, AnBook, Avr99}.  

Quillen \cite{Qu} conjectured that q.c.i.\ homomorphisms are the only maps that have 
finitely many non-vanishing Andr\'e-Quillen homology functors.  To devise a proof or 
to construct a counter-example, one needs to understand the structure of q.c.i.\ 
homomorphisms and/or to gain a detailed knowledge of their properties.

When $R$ is local J.~Majadas Soto and A.~Rodicio Garcia \cite{GS, SoJPAA, SoGMJ} 
proved that surjective q.c.i.\ homomorphisms are quasi-Gorenstein in the 
sense of \cite{AF:Gor}.  Rodicio conjectured that, up to faithfully flat base change, 
every q.c.i.\ homomorphism has a known form---it arises from a pair of 
embedded regular sequences; see \cite{SoJPAA}.

Here the goal is to systematically investigate the properties of general q.c.i.\ homomorphisms.  
Our results show that they are remarkably similar to those of l.c.i.\ maps, which form a much 
smaller class. Following a template devised in \cite{Avr99}--\cite{AFH} for studying
homomorphisms of noetherian rings, we proceed in three stages.  

The initial focus is on surjective homomorphisms of local rings.  For such maps we 
study the q.c.i.\ property by using a description in terms of the Koszul homology of 
$\Ker(\vf)$, due to Blanco, Majadas, and Rodicio \cite{BMR}.  Sections 
\ref{S:Quasi-complete intersection ideals} through \ref{S:Poincare series}
contain, among other things, short new proofs of the theorems of Garcia and Soto; 
formulas for the changes of depth and (with restrictions) of Krull dimension, 
which raise questions concerning modules of finite Gorenstein dimension; descriptions, 
in closed form, of the changes in the homological behavior of the residue fields;
examples of q.c.i.\ homomorphisms that do not have the form conjectured by Rodicio.  

At a second stage, the results on surjective maps of local rings are extended   
to arbitrary local homomorphisms.  This is done in Section \ref{S:Local homomorphisms 
of local rings} by utilizing Cohen factorizations of local homomorphisms, constructed in \cite{AFH}.

Finally, homomorphisms of noetherian rings are analyzed by patching the local results
through vanishing theorems for an appropriate (co)homology theory.  This is the content
of Section \ref{S:Homomorphisms of noetherian rings}, where the characterization
of q.c.i.\ homomorphisms in terms of Andr\'e-Quillen homology is used for the first time
in this paper.

\section{Quasi-complete intersection ideals}
  \label{S:Quasi-complete intersection ideals}

Throughout this section $(R,\fm,k)$ denotes a local ring; in detail: $R$ is a commutative 
noetherian ring with unique maximal ideal $\fm$, and $k=R/\fm$. In addition, $I$ is an 
ideal of $R$ and we set $S=R/I$.  We define quasi-complete intersection ideals, track 
their behavior under certain changes of rings, and provide examples.  

Let $\bd a$ be a finite generating set of $I$ and $E$ the Koszul complex 
on $\bd a$.  The homology $\HH{*}E$ has $\HH 0E=S$ and a structure
of graded-commutative $S$-algebra, inherited from the DG $R$-algebra $E$.
Thus, there is a unique homomorphism
  \begin{equation}
    \label{eq:lambda}
\lambda^S_*\colon \bw_*^S \HH 1E\lra\HH{*}E
  \end{equation}
of graded $S$-algebras with $\lambda_1^S=\id^{\HH1E}$, where $\bw_*^S$
is the exterior algebra functor.   

Note that $\lambda^S_*$ is bijective if (and only if) there exists \emph{some}
isomorphism of graded $S$-algebras $\lambda\col\HH{*}E\xra{\cong}\bw_*^S\HH 1E$.
Indeed, when this is the case the composed map
  \[
\bw_*^S \HH 1E\xra{\bw_*^S(\lambda_1)}\bw_*^S \HH 1E\xra{\,\lambda^{-1}\,}\HH{*}E
  \]
is an isomorphism, and it is equal to $\lambda^S_*$ because both maps restrict to $\id^{\HH1E}$.

  \begin{chunk}
    \label{def:qi}
We say that $I$ is \emph{quasi-complete intersection} if $\HH 1E$ is free over $S$ and 
$\lambda^S_*$ is bijective; this property does not depend on the choice of $\bd a$, see 
\cite[1.6.21]{BH}.  

To the best of our knowledge, the ideals defined above first appeared, with no denomination,
in Rodicio's paper \cite{RoCMH}; in his joint paper \cite{BMR0} with Blanco and Majadas, 
their defining property is called \emph{free exterior Koszul homology}.
  \end{chunk}

Recall that $\grade_RS$ denotes the maximal length of an $R$-regular 
sequence in $I$, and that this number is equal to the least integer $n$ with 
$\Ext nRSR\ne0$.  

  \begin{lemma}
    \label{lem:principal}
If $I$ is quasi-complete intersection, then 
  \begin{align}
  \label{eq:grade1}
\grade_RS&=\rank_R E_1-\rank_S\HH1E\,.
 \end{align}
  \end{lemma}

\begin{proof}
The grade-sensitivity of $E$, see \cite[1.6.17(b)]{BH}, yields 
  \[
\rank_R E_1 - \grade_RS =\max\{n\mid\HH nE\ne0\}\,.
  \]
As $\HH{*}E$ is isomorphic to $\bw_*^S\HH1E$, the right-hand side equals $\rank_S\HH1E$.
 \end{proof}

Quasi-complete intersection ideals are stable under certain base changes.

  \begin{lemma}
    \label{lem:qciflat}
Let $R'$ be a local ring and $\rho\colon R\to R'$ a flat homomorphism of rings. 

When $IR'\ne R'$ holds, if $I$ is quasi-complete intersection, then so is $IR'$.  

When $\rho$ is faithfully flat, if $IR'$ is quasi-complete intersection, 
then so is $I$.
  \end{lemma}

  \begin{proof}
Note that $R'\otimes_RE$ 
is the Koszul complex on the generating set $\rho(\bd a)$ of $IR'$.
As $R'$ is a flat $R$-module, there is a natural isomorphism
$\HH *{R'\otimes_RE}\cong R'\otimes_R\HH *{E}$ of graded 
algebras, whence the assertions follow by standard arguments.
  \end{proof}
  
  \begin{lemma}
    \label{lem:step}
Let $\bd b$ be an $R$-regular sequence contained in $I$.

The ideal $I$ of $R$ is quasi-complete intersection if and only if the ideal 
$\ov I=I/(\bd b)$ of $\ov R=R/(\bd b)$ is quasi-complete intersection.
  \end{lemma}

\begin{proof}
By induction, we may assume ${\bd b}=b$.  Let $E$ be the Koszul complex
on a generating set $\{a_1,\dots,a_c\}$ of $I$ with $a_1=b$.  Pick
$v_1,\dots,v_c$ in $E_1$ with $\dd(v_i)=a_i$ for $i=1,\dots,c$, set $v=v_1$, and 
note that $J:=(b,v)E$ is a DG ideal.  Each $y\in J$ can be written  as 
$y=be+vf$ with $e,f\in E'$, where $E'$ is the DG $R$-subalgebra of $E$, 
generated by $v_2,\dots,v_c$.  Now $\dd(y)=0$ implies $b(\dd(e)+f)=0$.  
As $b$ is regular, this gives $f=-\dd(e)$, hence $y=\dd(ve)$.  We proved $\HH{*}J=0$,
so $E\to E/J$ induces $\HH{*}E\cong\HH{*}{E/J}$ as graded $S$-algebras.  
It remains to note that $E/J$ is isomorphic to the Koszul complex on ${\bd a}=
\{a_2+(\bd b),\dots,a_c+(\bd b)\}\subseteq\ov R$, and $\ov I=({\bd a})$.
 \end{proof}

In our study of quasi-complete intersection ideals a fundamental role is played
by a classical construction of Tate \cite{Ta}; see also \cite[\S1.1]{GL}, \cite[\S6]{Avr98}.

\begin{construction}
 \label{con:tate}
Let $\bd a=\{a_1,\dots,a_c\}$ be a generating set of $I$ and $E$ the 
Koszul complex on $\bd a$.  Fix a basis $\bd v=\{v_1,\dots,v_c\}$ of $E_1$ 
with $\dd(v_i)=a_i$ for $1\le i\le c$ and a set of cycles $\bd z=\{z_1,\dots,z_h\}$ 
whose classes minimally generate $\HH1E$.  

Let $W$ be a graded $R$-module that has a basis $\bd w=\{w_1,\dots,w_h\}$ of
elements of degree $2$, and let $\Gam^R_*W$ be the divided powers
algebra on $W$.  The products $w_1^{(j_1)}\cdots w_h^{(j_h)}$ with $j_i\ge0$
and $j_1+\cdots+j_h=p$ form an $R$-basis of $\Gam^R_pW$.

The \emph{Tate complex} on $E$ and $\bd z$ is the complex $F$ of free
$R$-modules with
  \begin{alignat}{2}
    \label{eq:tatec}
F_n&=&\bigoplus_{2p+q=n}\Gam^R_{p}W&\otimes_RE_q\,.
   \\
    \label{eq:tated}
\dd^F_n(w_1^{(j_1)}\cdots w_h^{(j_h)}\otimes e) &=&\ w_1^{(j_1)}\cdots
w_h^{(j_h)}&\otimes\dd^E(e) \\
  \notag
&&+\sum_{i=1}^h w_1^{(j_1)}&\cdots w_i^{(j_i-1)}\cdots w_h^{(j_h)}\otimes
z_{i}e\,.
  \end{alignat}

It is clear from the construction that $\HH0F=S$ holds.  
  \end{construction}

  \begin{chunk}
    \label{ch:tate}
Let $F$ be the complex in Construction \ref{con:tate}.  From \cite[Theorems 4]{Ta} one gets:

If $I$ is quasi-complete intersection, then $F$ is a \emph{free resolution} of the 
$R$-module $S$.  If $\bd a$ generates $I$ minimally, then $F$ is \emph{minimal}: each 
$z_i$ then is a syzygy in the free cover $E_1\to I$, whence $z_i\in\fm E_1$, so 
$\dd(F)\subseteq\fm F$ holds by \eqref{eq:tated}.
  \end{chunk}

We finish the section with descriptions of quasi-complete intersection ideals in terms 
of the homology functors $\AQ nSR-$ of Andr\'e \cite{AnBook} and Quillen \cite{Qu}.  Until
the last section, these are used solely for comparisons with existing results or arguments. 
 
\begin{chunk}
     \label{ch:aq}
The ideal $I$ is quasi-complete intersection\ if and only if $\AQ nSR-=0$ holds for $n\ge3$, 
if and only if $\AQ nSRk=0$ holds for $n\ge3$.

The first equivalence is due to Blanco, Majadas, and Rodicio and appears in \cite[Corollary~3$'$]{BMR};
the proofs of similar statements in \cite{RoCMH} and \cite{BMR0} are incomplete, see 
\cite[pp.\ 126--127]{BMR}.  The second equivalence follows from \cite[4.57]{AnBook}.
  \end{chunk}

\section{Between complete intersection and Gorenstein}
  \label{S:Between complete intersection and Gorenstein}

In this section $I$ is an ideal in a local ring $(R,\fm,k)$.  We compare the 
quasi-complete intersection property and other properties of ideals, which we recall next. 

  \begin{chunk}
    \label{ex:ci}
Recall that $I$ is said to be \emph{complete intersection} if it has a
generating set $\bd a$ satisfying the following equivalent conditions: (i) $\bd a$ is
an $R$-regular sequence; (ii) $\HH1E=0$; (iii) $\HH nE=0$ for all $n\ge1$.
Thus, one evidently has:

Every complete intersection ideal is quasi-complete intersection.

The ideal $\fm$ is complete intersection if and only if $R$ 
is regular; see \cite[2.2.5]{BH}.
  \end{chunk}

 \begin{chunk}
    \label{ch:ciRing}
By Cohen's Structure Theorem, the $\fm$-adic-completion of $R$ 
has a presentation $\wh R\cong Q/J$, with $Q$ a regular 
local ring.  The ring $R$ is said to be \emph{complete intersection} if in 
some Cohen presentation of $\wh R$ the ideal $J$ is generated
by a regular sequence; this property is independent of the presentation, 
see \cite[19.3.2]{Gr} or \cite[2.3.3]{BH}.  

The ideal $\fm$ is quasi-complete intersection if and only if the ring $R$ 
is complete intersection;  this is due to Assmus, \cite[2.7]{As}, see 
also \cite[2.3.11]{BH}. 
  \end{chunk}

  \begin{chunk}
    \label{ch:cidim}
A \emph{quasi-deformation} is a pair $R\to R'\gets Q$ of homomorphisms
of local rings, where $R\to R'$ is faithfully flat and $R'\gets Q$ is surjective with 
kernel generated by a $Q$-regular sequence.  By definition, the \emph{CI-dimension}
of an $R$-module $M$, denoted $\cidim_RM$, is finite if $\pd_Q(R'\otimes_RM)$ 
is finite for some quasi-deformation; see~\cite{AGP2}.

If $R$ is complete intersection, then $\cidim_RM$ is finite for each $M$.  

If $\cidim_R\fm$ is finite, then $R$ is complete intersection; see 
\cite[1.3, 1.9]{AGP2}.
  \end{chunk}

  \begin{chunk}
    \label{ch:gorRing}
We say that $I$ is \emph{quasi-Gorenstein} if $\Ext{\grade_RS}RSR\cong S$ 
and $\Ext nRSR=0$ for $n> \grade_RS$; thus, $\fm$ is quasi-Gorensten if 
and only if $R$ is a Gorenstein ring.

By \cite[2.3]{AI}, the ideal $I$ is Gorenstein if and only if 
the homomorphism $R\to R/I$ is quasi-Gorenstein in the sense of 
\cite{AF:qGor}.  In particular, \cite[6.5, 7.4, 7.5]{AF:qGor} yield
  \begin{align}
  \label{eq:cmd1}
\grade_RS&=\depth R-\depth S\,.
  \end{align}
A short proof of this equality, following \cite[Corollary\,5]{GS}, proceeds as follows:

Set $d=\depth S$ and $g=\grade_RS$.  As $\Ext gR{S}R\cong S$, the spectral sequence 
  \[
\EC 2{p,q}=\Ext p{S}k{\Ext qR{S}R}\implies \Ext {p+q}RkR
  \]
yields $\Ext{p}{R}{k}{R}=0$ for $p<d+g$, and $\Ext{d+g}RkR\cong\Ext d{S}{k}{S}\ne0$.
  \end{chunk}

We present new, direct proofs of two key results of Garcia and Soto,
comparing the quasi-complete intersection condition with other properties:  
The implication (3) in the next theorem is \cite[Proposition\,12]{SoJPAA} 
and (4) is obtained in \cite[Remark\,8]{GS}.

\begin{theorem}
  \label{thm:hierarchy}
For an ideal $I$ of a local ring $R$, the ring $S=R/I$, and the Koszul complex
$E$ on some set of generators for $I$ the following implications hold
 \[
\xymatrixrowsep{.1pc}
\xymatrixcolsep{.8pc}
\xymatrix{
\text{$I$ is complete intersection }
\ar@{<=>}[r]^-{\textstyle(1)}
&
\text{\begin{tabular}{c}
$I$ is quasi-complete intersection\\
and $\pd_RS$ is finite
\end{tabular}}
\ar@{=>}[ddd]_-{\textstyle(2)\ }
  \\
  \\
  \\
\text{$I$ is quasi-complete intersection }
\ar@{=>}[ddd]_-{\textstyle(4)\ }
&
\text{\begin{tabular}{c}
$\HH1E$ is free over $S$\\
and $\cidim_RS$ is finite
\end{tabular}}
\ar@{=>}[l]_-{\textstyle(3)}
  \\
  \\
  \\
\text{$I$ is quasi-Gorenstein}
  }
   \]
 \end{theorem}

\begin{remark} 
The simple implications in the theorem are irreversible:
For (2), apply \ref{ex:ci} and \ref{ch:ciRing} with $R=k[x]/(x^2)$ and $I=\fm$.  For (4), apply
\ref{ch:ciRing} and \ref{ch:gorRing} with $R=k[x,y,z]/(x^2-y^2,y^2-z^2,xy,xz,yz)$ and $I=\fm$.  
For (3), see Theorem~\ref{thm:agp}.
   \end{remark}

To prove (3) we use a lemma that can be extracted from the argument for \cite[Proposition\,12]{SoJPAA},
which uses Andr\'e-Quillen homology.  Here is a short direct proof.

\begin{lemma}
  \label{lem:hierarchy}
(Notation as in Theorem \emph{\ref{thm:hierarchy}})  When $\cidim_RS$ is finite and 
$\HH1E$ is free over $S$ there exist a local ring $Q$, complete intersection ideals 
$J\subseteq I'$ of $Q$, and a flat local homomorphism $R\to Q/J$ such that $I(Q/J)=I'/J$.  
  \end{lemma}

   \begin{proof}
Let $R\to R'\gets Q$ be a quasi-deformation with $\pd_Q(R'\otimes_RS)$ finite,
set $S'=R'\otimes_RS$ and $I'=\Ker(Q\to R')$.  Let $E'$ be the Koszul complex on 
a subset $\bd a'$ of $Q$ with $R'\bd{a}'=I'$ and $B$ the Koszul complex on a 
$Q$-regular set $\bd b$ generating $I'$.  The
quasi-isomorphism $B\to R'$ induces a quasi-isomorphism $B\otimes_QE'
\to R'\otimes_QE'$.  It 
gives the first isomorphism in the following string, where the second one is induced by 
$R'\otimes_QE'\cong R'\otimes_RE$ and the third one is due to the flatness of $R'$ over $R$:
  \[
\HH1{B\otimes_QE'}\cong\HH1{R'\otimes_QE'}\cong\HH1{R'\otimes_RE}\cong
R'\otimes_R\HH1{E} \cong S'\otimes_S\HH1{E}
  \]
Thus, $\HH1{B\otimes_QE'}$ is free over $S'$.  As $B\otimes_QE'$ is the Koszul complex on 
$\bd b\sqcup\bd a'$ and this set generates $J$, Gulliksen \cite[1.4.9]{GL} shows that $J$ is 
complete intersection.
  \end{proof}

In \cite{GS} the implication (4) in Theorem \ref{thm:hierarchy} is derived from a result 
concerning algebras with Poincar\'e duality.  Our proof uses the DG $E$-module 
structure of Tate complexes.  We describe some notions used to handle it.

  \begin{chunk}
    \label{ch:DGE}
Let $X$ be a complex of $R$-modules.  We let $|x|=n$ stand for $x\in\ X_n$.   

The complex $\shift X$ has $(\shift X)_n=X_{n-1}$ and $\dd^{\shift X}_n=-\dd^X_{n-1}$;
for $n\in\mathbb Z$;  for $x\in X_{n-1}$ we let $\varsigma(x)$ denote the element 
$x\in(\shift X)_n$.  When $X$ is a DG $E$-module, so is $\shift X$, with action of $E$ given by 
$e\varsigma(x)=(-1)^{i}\varsigma(ex)$ for $e\in E_i$.

The complex $X^{\vee}$ has $(X^{\vee})_n=\Hom R{X_{-n}}R$ and
$(\dd^{X^{\vee}}(\chi))(x)=(-1)^{n}(\dd^{X}(x))$ for $\chi\in(X^{\vee})_n$, $x\in X_{-n-1}$,
and $n\in\mathbb Z$.  When $X$ is a DG $E$-module, so is $X^{\vee}$, with action of $E$
given by $(e\cdot\chi)(x)=(-1)^{|i||\chi|}\chi(ex)$ for $e\in E_i$.
  \end{chunk}

\stepcounter{theorem}

 \begin{proof}[Proof of Theorem \emph{\ref{thm:hierarchy}}]
(1)  If $I$ is quasi-complete intersection and $\pd_RS$ is finite,
then $F_n=0$ for $n\gg0$ in the minimal resolution $F$ of 
\ref{ch:tate}.  This forces $W=0$, hence $\HH1E=0$,
so $I$ is complete intersection.  The converse is clear.

(2) This implication holds because finite projective dimension implies 
finite CI-dimension, as testified by the quasi-deformation 
$R\xra{\,=\,}R\xla{\,=\,}R$.

(3)  This follows directly from Lemmas \ref{lem:hierarchy} and \ref{lem:qciflat}.

(4) We use the notation from Construction \ref{con:tate}.  Also, when $X$ is
a complex of $R$-modules, $\shift X$ denotes its translation.  For elements
of $X$, we let $|x|=n$ stand for $x\in\ X_n$.  Let $X^{\vee}$ denote
the complex with $(X^{\vee})_n=\Hom R{X_{-n}}R$ for $n\in\mathbb Z$, and
$(\dd^{X^{\vee}}(\chi))(x) =(-1)^{|\chi|}(\dd^{X}(x))$ for $\chi\in X^{\vee}$ and
$x\in X$.  When $X$ is a DG $E$-module, so is $X^{\vee}$, with action of $E$
given by $(e\cdot\chi)(x)=(-1)^{|e||\chi|}\chi(ex)$ for $e\in E$.

Let  $\bd\omega=\{\omega_1,\dots,\omega_h\}$ be the basis of $W^{\vee}$, dual to 
$\bd w$, and $\Sym^R_*(W^{\vee})$ the symmetric algebra of $W^{\vee}$.  Standard 
isomorphisms of $R$-modules, $(\Gam^R_pW)^{\vee}\cong\Sym^R_p(W^{\vee})$,
take the elements of the basis of $(\Gam^R_pW)^{\vee}$, dual to the basis of
$\Gam^R_pW$ described in Construction \ref{con:tate}, to the corresponding
products $\omega_1^{j_1}\cdots \omega_h^{j_h}$.  Thus, we have
  \begin{align}
(F^{\vee})_n&=\bigoplus_{2p+q=n} \Sym^R_{-p}(W^{\vee})\otimes_R(E^{\vee})_q\,.
  \\
    \label{eq:tateDuals}
\dd^{F^{\vee}}(\omega\otimes\epsilon)
&=\omega\otimes\dd^{E^{\vee}}(\epsilon)+\sum_{i=1}^h\omega_i\omega\otimes z_i\cdot\epsilon\,.
  \end{align}

A decreasing filtration of $F^{\vee}$ is given by the $R$-linear spans of the sets
  \[
\{(\omega_1^{j_1}\cdots \omega_h^{j_h}\otimes\epsilon)\in F^{\vee}\mid j_i\ge0\,,
j_1+\cdots+j_h\ge-p\,, \epsilon\in E^{\vee}\}
  \]
for $p\le0$. The resulting spectral sequence $\EE r{p,q}{\implies}\HH{p+q}{F^{\vee}}$ 
starts with
  \begin{equation}
    \label{eq:E0}
\EE 0{p,q}=\Sym^R_{-p}(W^{\vee})\otimes_R(E^{\vee})_{q-p}
  \quad\text{and}\quad
\Ed0{p,q}=\Sym^R_{-p}(W^{\vee})\otimes_R\dd^{E^{\vee}}_{q-p}\,.
  \end{equation}

Let $U$ be an $R$-module with basis $\bd u=\{u_1,\dots,u_h\}$.   The
hypothesis on $I$ yields an isomorphism $S\otimes_R\bw_*^RU\cong\HH{*}E$
of graded $R$-algebras sending $u_i$ to the class of $z_i$ for $i=1,\dots,h$. 
Let $\tau^{\bd u}\in(\bw_*^RU)^{\vee}_{-h}$ be the $R$-linear map with
$\tau^{\bd u}(u_1\cdots u_h)=1$; the map $u\mapsto u\cdot\tau^{\bd
u}$ is an isomorphism $\bw_*^RU\cong\shift^h(\bw_*^RU)^{\vee}$ of graded
$\bw_*^RU$-modules.  Similarly, the $R$-linear map $\tau^{\bd v}\in
(E^{\vee})_{-c}$, defined by $\tau^{\bd v}(v_1\cdots v_c)=1$, yields an
isomorphism  of DG $E$-modules $E\cong\shift^{-c}(E^{\vee})$, given by
$e\mapsto e\cdot\tau^{\bd v}$.  Thus, we get
  \begin{equation}
    \label{eq:isoss}
\HH{q-p}{E^{\vee}}
\cong\HH{q-p+c}E
\cong S\otimes_R\bw^R_{q-p+c}U
\cong S\otimes_R(\bw_*^RU)^{\vee}_{q-p+c-h}
  \end{equation}
as $R$-modules.  Now \eqref{eq:E0} and \eqref{eq:isoss} yield
   \begin{align*}
 \EE1{p,q}=\HH{q}{\EE0{p,{\scriptscriptstyle\bullet}}}
 &\cong\Sym^R_{-p}(W^{\vee})\otimes_R\HH{q-p}{E^{\vee}}
   \\
 &\cong S\otimes_R(\Sym^R_{-p}(W^{\vee})\otimes_R(\bw_*^RU)^{\vee}_{q-p+c-h})
   \\
 &\cong S\otimes_R(\Gam^R_{-p}(W)\otimes_R \bw^R_{p-q+(h-c)}U)^{\vee}\,.
  \end{align*}

Let $C$ be the complex with $C_n=\bigoplus_{2p+q=n}\EE1{p,q}$
and $\dd^C_n=\bigoplus_{2p+q=n}\Ed1{p,q}$.  Let $G$ be the Tate
complex on the Koszul complex $\bw_*^RU$ with zero differentials
and the set of cycles $\bd u$.  Formulas \eqref{eq:tateDuals} and
\eqref{eq:tated} show that the maps above produce an isomorphism of
complexes $C\cong S\otimes_R(\shift^{c-h}G)^{\vee}$.  By \ref{ch:tate}, $G$
is an $R$-free resolution of $R$, so $G$ is homotopically equivalent
to $R$.  Consequently, $(\shift^{c-h}G)^{\vee}$ is homotopically equivalent to
$\shift^{h-c}R$, whence $\HH nC=0$ for $n\ne h-c$, and $\HH{h-c}C\cong S$.
We have $C_{h-c}=\EE1{0,h-c}$ so the computation of $\HH{*}C$ yields
  \[
\EE2{p,q}=\HH{p}{\EE1{{\scriptscriptstyle\bullet},q}}\cong
\begin{cases}
 S &\text{for }(p,q)=(0,h-c)\,,\\
 0 &\text{otherwise}\,.
\end{cases}
  \]

The convergence of the spectral sequence implies $\HH{h-c}{F^{\vee}}\cong S$ and
$\HH{n}{F^{\vee}}=0$ for $n\neq h-c$.  One has $\HH{n}{F^{\vee}}\cong\Ext{-n}RSR$,
see \ref{ch:tate}, hence $\Ext{c-h}RSR\cong S$ and $\Ext
nRSR=0$ for $n\ne c-h$, so $I$ is quasi-Gorenstein.
  \end{proof}

\section{Exact ideals}
  \label{S:Exact ideals}

In this section $(R,\fm,k)$ denotes a local ring.  We first discuss principal quasi-complete 
intersection ideals.  They are easy to identify, are amenable to explicit computations, and 
have turned up in abundance in recent studies.  Exact ideals are defined in terms of 
principal quasi-complete intersection ideals by analogy with the way that complete 
intersection ideals are formed from regular elements.

  \begin{chunk}
    \label{rem:ezdqci}
Let $N$ be an $R$-module.  We say that an element $a$ of $R$ is an 
\textit{exact zero-divisor} on $N$ if there exists an element $b\in R$, such that 
satisfies
  \begin{equation}
    \label{eq:ezd}
N \ne (0:a)_N\cong aN\ne 0\,.
 \end{equation}
or, equivalently, $aN=(0:b)_N\ne0$ and $bN=(0:a)_N\ne0$.  This shows that $b$ 
also is an exact zero-divisor on $N$ and that the following sequence is exact:
  \begin{equation}
    \label{eq:periodic}
\cdots \to N\xrightarrow{b}N\xrightarrow{a}N\xrightarrow{b}N\xrightarrow{a}N\to\cdots
 \end{equation}
Thus, $(a,b)$ is an \emph{$N$-exact pair} in the sense of Kie{\l}pi\'nski, 
Simson, and Tyc \cite[1.1]{KST}.  

Following \cite[1.2]{HS}, when $N=R$ we speak simply of exact zero-divisors. 

We say that an element $a$ of $R$ is \emph{exact} if it is $R$-regular or an exact 
zero-divisor.  This holds if and only if $(a)$ is quasi-complete intersection:
The Koszul complex $E$ on $\{a\}$ has $E_{\ges2}=0$ and $\HH1E=(0:a)_R$, 
so $\HH1E$ is free over $R/(a)$ and $\lambda^{R/(a)}_*$ in \eqref{eq:lambda} is 
bijective if and only if $(0:a)_R=0$ or $(0:a)_R\cong {R/(a)}$.  
   \end{chunk}
 
In some cases, only half of the conditions in \eqref{eq:ezd} need verification. 

\begin{lemma}
 \label{lem:exact_cm}
Let $N$ be an $R$-module of finite length.

If $(0:b)_N=aN$ holds for some elements $a,b\in R$, then one has
$(0:a)_N=bN$.
 \end{lemma}

  \begin{proof}
{From} $a(bN)=b(aN)=b(0:b)_N=0$ one gets $(0:a)_N\supseteq  bN$.  A length 
count, using this inclusion and the composition $N/aN=N/(0:b)_N\cong bN$,
gives 
\begin{align*}
\ell(N)-\ell(aN)
=\ell((0:a)_N)
\ge\ell(bN)
=\ell(N/aN)
=\ell(N)-\ell(aN)\,.
\end{align*}
These relations imply $\ell((0:a)_N)=\ell(bN)$, hence $(0:a)_N=bN$.
  \end{proof}
 
For Cohen-Macaulay rings the situation is only slightly more complicated.

\begin{proposition}
For an element $a$ of $\fm$ the following conditions are equivalent.
\begin{enumerate}[\quad\rm(i)]
\item 
$a$ is an exact zero-divisor and $R$ is Cohen-Macaulay.
\item
$(0:a)_R=(b)\ne R$ for some $b$ in $R$, and $\depth R/(a)\ge\dim R$.
\item 
$(0:a)_R=(b)\ne R$ for some $b$ in $R$, and $R/(a)$ is Cohen-Macaulay
with $\dim R/(a)=\dim R$.
  \end{enumerate}
\end{proposition}

\begin{proof}
(i)$\implies$(ii).  Using \eqref{eq:cmd1}, we get $\depth R/(a)=\depth R=\dim R$.

(ii)$\implies$(iii).  Use the inequalities $\dim R\ge\dim R/(a)\ge\depth R/(a)$.

(iii)$\implies$(i). 
Set $K=(0:b)_R/(a)$ and pick $\fp\in\Ass_R(K)$.  From
$K\subseteq R/(a)$ we get $\fp\in\Ass_R R/(a)$.  The
Cohen-Macaulayness of $R/(a)$ implies that $\fp$ is minimal in $\Supp_R
R/(a)$ and satisfies $\dim(R/\fp)=\dim R/(a)$.  We have $\dim R/(a)=\dim
R$, so $\fp$ is minimal in $\Spec R$, hence the ring $R$ is artinian.
As one has $(0:(a/1))_{R_\fp}=(b/1)R_{\fp}$, Lemma \ref{lem:exact_cm}
yields $(a/1)R_\fp=(0:(b/1))_{R_{\fp}}$; that is, $K_{\fp}=0$.  Since $\fp$ can 
be chosen arbitrarily in $\Ass_R(K)$, we conclude that $K$ is equal to zero.

Thus, $a$ is an exact zero-divisor.  Now \eqref{eq:cmd1} and the hypotheses 
yield
  \[
\depth R=\depth R/(a)=\dim R/(a)=\dim R\,.\qedhere
  \]
  \end{proof}

  \begin{remark}
The implication (iii)$\implies$(i) may fail when $\dim R/(a)\ne\dim R$.

Indeed, set $R=k[x,y]/(xy,y^2)$, and let $a$ and $b$ denote the images
in $R$ of $x$ and $y$, respectively.   The equality $(0:a)_R=(b)$ shows
that $(0:a)_R$ is principal, and the isomorphism $R/(a)\cong k[y]/(y^2)$
that $R/(a)$ is Cohen-Macaulay.  However, $R$ is not Cohen-Macaulay;
neither is $a$ an exact zero-divisor, as $(0:b)_R=(a,b)\ne(a)$.
  \end{remark}

Now we settle, in the negative, a conjecture of Rodicio; see \cite[Conjecture\,11]{SoJPAA}.

\begin{theorem}
    \label{thm:agp}
If $k$ is a field of characteristic different from $2$, then in the ring
  \[
R=\frac{k[w,x,y,z]}{(w^2,\, wx-y^2\,, wy-xz,\,wz,\, x^2+yz, xy,\, z^2)}
  \]
the ideal $xR$ is quasi-complete intersection and has infinite CI-dimension. 
   \end{theorem}

  \begin{proof}
Let $a$ and $b$ denote the images in $R$ of $x$ and $y$, respectively;
thus, $xR=(a)$.  A simple calculation yields $(0:a)_R=(b)$ and $(0:b)_R=(a)$, 
so $a$ is an exact zero-divisor, and hence $(a)$ is a quasi-complete 
intersection ideal; see~\ref{rem:ezdqci}.  

Assume $\cidim_RI$ is finite.  Yoneda products then turn $\Ext *R{R/(a)}k$ 
into a finite graded module that is finite over a $k$-subalgebra of $\Ext*Rkk$, 
generated by central elements of degree~$2$; see \cite[5.3]{AS}.  On the 
other hand, in \cite[p.~4]{AGP1} it is shown that $\Ext 2Rkk$ contains no 
non-zero central element of $\Ext *Rkk$, hence $\Ext nR{R/(a)}k=0$ for all 
$n\gg0$.  This is impossible, since \eqref{eq:periodic} with $N=R$ gives a 
free resolution of $R/(a)$ that yields $\Ext nR{R/(a)}k\cong k$ for $n\ge0$.
   \end{proof}

In \cite[1.1]{KST} \emph{exact sequences of pairs} are defined by analogy with 
regular sequences of elements.  Both notions are subsumed into the next concept.  

  \begin{chunk}
    \label{exact}
We say that a sequence $a_1,\dots,a_c$ in $R$ is \emph{exact} if  
$a_i$ in exact on $R/(a_1,\dots,a_{i-1})$ for $i=1,\dots,c$. 
An ideal that can be generated by an exact sequence is called \emph{exact}.
  \end{chunk}

  \begin{theorem}
    \label{thm:exact}
Let $a_1,\dots,a_c$ be an exact sequence in $R$ and set $I=(a_1,\dots,a_c)$.

The ideal $I$ then is quasi-complete intersection, and for $S=R/I$ one has
   \[
\grade_RS=\card\{i\in[1,c]\mid \text{$a_i$ is regular on $R/(a_1,\dots,a_{i-1})$}\}\,.
  \]
  \end{theorem}

  \begin{proof}
For $c=1$ see \ref{rem:ezdqci}, so fix $c\ge2$.  Set $I'=(a_1,\dots,a_{c-1})$ and $R'=R/I'$.  Let $E'$ 
and $E$ be the Koszul complexes on $\{a_1,\dots,a_{c-1}\}$ and $\{a_1,\dots,a_{c}\}$, respectively.
By induction, we assume $\HH{*}{E'}\cong\bw_*^{R'}(\HH1{E'})$ as algebras and $\HH1{E'}$ is $R'$-free.

Considering $E'$ as a DG subalgebra of $E$, set $a=a_c$, and choose $e\in E_1$ with $\dd(e)=a$.  
Thus, the elements of $E$ have the form $y'+ey''$ with unique $y',y''\in E'$, and $E/E'\cong\shift E'$
as DG $E$-modules; see \ref{ch:DGE}.  The exact sequence of DG $E$-modules
  \[
0\to E'\to E\to E/E'\to0
  \]
induces in homology an exact sequence of $R'$-modules
  \[
\cdots\to\HH n{E'}\xra{\,a\,}\HH{n}{E'}\to\HH nE\to\HH{n-1}{E'}\xra{\,a\,}\HH{n-1}{E'}\to\cdots
   \]
and hence an exact sequence of graded $\HH{*}{E'}$-modules
  \begin{equation}
    \label{eq:exact2}
0\to\HH{*}{E'}/a\HH{*}{E'}\to\HH{*}E\to\shift(0:a)_{\HH{*}{E'}}\to0
  \end{equation}

Let $H'$ be the image $H'$ of $\HH{*}{E'}/a\HH{*}{E'}$ in $\HH{*}E$.  Note that it is a graded $S$-subalgebra, 
with $H'\cong\HH{*}{E'}/a\HH{*}{E'}$ as graded algebras and $H'_1$ free over $S$.

If $a$ is $R'$-regular, then $(0:a)_{\HH{*}{E'}}=0$, so $\HH{*}E=H'$ by \eqref{eq:exact2}.  This
implies $\HH{*}E\cong\bw_*^S(\HH1E)$, and  hence $\grade_RS=\grade_RR'+1$ by \eqref{eq:grade1}.

If $a$ is an exact zero-divisor, we have $(0:a)_{\HH{*}{E'}}\cong\HH{*}{E'}/a\HH{*}{E'}$ because $\HH{*}{E'}$ is 
$R'$-free.  Thus, \eqref{eq:exact2} gives $\HH{*}E\cong H'\oplus \shift H'$ as graded $H'$-modules.  It follows 
that there exists an element $h$ in $\HH1E$, such that $\HH1E=H'_1\oplus R'h$ and $h'h=(-1)^{|h'|}hh'$ holds for $h'\in H'$.  
Since $h^2=0$ we get an isomorphism of 
algebras $\HH{*}E\cong\bw_*^{S}(\HH1E)$, and then $\grade_RS=\grade_RR'$ follows from \eqref{eq:grade1}.
  \end{proof}

  \begin{remark}
    \label{counterex}
Theorem \ref{composition}, proved by using Andr\'e-Quillen homology, gives an independent 
proof that exact ideals are quasi-complete intersection.  In an earlier version of this 
paper we had asked whether the converse holds.  Since then, a negative answer 
has been obtained by Kustin, \c Sega, and Vraciu~\cite{KSV}.
  \end{remark}

We finish with examples of exact zero-divisors, taken from the existing literature.

\begin{example}
  \label{ex:generic1}
Let $k$ be an algebraically closed field and $\ul R$ a graded ring of the form 
$k[x_1,\dots,x_e]/\ul I$, with $\deg(x_i)=1$ and $\ul I$ generated by forms of 
degree at least $2$.

The quadratic $k$-algebras $\ul R$ with $\rank_k\ul R_2=e-1$ are
parametrized by the points of the Grassmannian of subspaces of rank
$\binom{e}2+1$ in the $\binom{e+1}2$-dimensional affine space over $k$.
Conca \cite[Theorem\,1]{Co} proved that this Grassmannian contains
a non-empty open subset $U$, such that each $\ul R$ in $U$ contains
an element $a\in\ul R_1$ with $a^2=0$ and $a\ul R_1=\ul R_2\ne0$.
This implies $\rank_k(a)=e$ and $(a{\ul R})\subseteq (0:a)_{\ul R}$.
A~length count shows that equality holds, so $a$ is an exact zero-divisor.
 \end{example}

\begin{example}
  \label{ex:generic2}
Let $k$ and $\ul R$ be as in the preceding example

The Gorenstein $k$-algebras $\ul R$ with $\ul R_4=0\ne \ul R_3$ are
parametrized by the points of the $\binom{e+2}3$-dimensional projective
space over $k$.  This space contains a non-empty open subset $U$, such
that each $\ul R\in U$ contains some exact zero-divisor $a\in\ul R_1$:
this follows from Conca, Rossi, and Valla \cite[2.13]{CRV}; see also
\cite[3.5]{HS}.
 \end{example}

  \section{Grade}
    \label{S:Grade}

In this section $(R,\fm,k)$ denotes a local ring, $I$ is an ideal of $R$, and $S=R/I$.

Recall that $(R,\fm,k)$ is said to be \emph{quasi-homogeneous}
if there is a finitely generated graded $k$-algebra $\ul
R=\bigoplus_{i=0}^\infty{\ul R}_i$ with $\ul R_0=k$, such that
the $\fm$-adic-completion $\wh R$ of $R$ is isomorphic to the
$(\bigoplus_{i=1}^\infty\ul R_i)$-adic completion $\wh{\ul R}$ of $\ul R$.
The ideal $I$ is called quasi-homogeneous if $I={\ul I}\wh R$
for some homogeneous ideal ${\ul I}$ of $\ul R$.

   \begin{theorem}
    \label{thm:dim1}
If $I$ is quasi-complete intersection, then 
  \begin{equation}
  \label{eq:cmd2}
\grade_RS=\depth R-\depth S\,.
 \end{equation}

Furthermore, an equality 
  \begin{equation}
  \label{eq:dim1}
\grade_R S=\dim R-\dim S
 \end{equation}
holds when one of the following conditions is satisfied:
  \begin{enumerate}[\quad\rm(a)]
 \item
$\cidim_RS$ is finite.
 \item
$I$ is exact.
 \item
$I$ is quasi-homogeneous.
  \end{enumerate}
 \end{theorem}

\begin{proof} 
In view of Theorem \ref{thm:hierarchy}(4), the equality \eqref{eq:cmd2} follows 
from \eqref{eq:cmd1}.

For \eqref{eq:dim1} we give a different argument in each case.

(a) Lemma \ref{lem:hierarchy} gives a local ring $Q$, complete intersection 
ideals $J\subseteq I'$ of $Q$, and a flat local homomorphism $R\to Q/J$ 
with $I(Q/J)=I'/J$.  For $R'=Q/J$ and $S'=S\otimes_RR'$ we have 
$\dim R-\dim S=\dim R'-\dim S'$ because $S\to S'$ is a flat local 
homomorphism with $k\otimes_SS'\cong k\otimes_RR'$.  As $J$ and $I'$ 
are generated by regular sequences we have $\dim R'-\dim S'=\grade_{R'}S'$.  
Flatness yields for each $n$ an isomorphism $\Ext n{R'}{S'}{R'}\cong
\Ext nRSR\otimes_RR'$, hence $\grade_{R'}S'=\grade_RS$.

(b)  By hypothesis, $I=(a_1,\dots,a_c)$ for some exact sequence $a_1,\dots,a_c$.  
Set $d=\dim R$ and let $a$ be an exact element of $R$.  If $a$ is regular, then
$\dim R/(a)=d-1$, so in view of Theorem \ref{thm:exact} it suffices to show that 
$\dim R/(a)=d$ holds when $a$ is an exact zero-divisor.
This is evident for $d=0$, so we further assume $d\ge1$.  

Suppose, by way of contradiction, that $\dim R/(a)<d$ holds.  
Choose $\fq\in\Spec R$ with $\dim(R/\fq)=d$.  We must have $a\not\in\fq$, whence
  \[
\dim  R/(a)\le d-1=\dim R/(\fq+(a))\le \dim R/(a)\,.
  \]
Thus, some prime ideal $\fp$ containing $\fq+(a)$ satisfies $\dim (R/\fp)=\dim
R/(a)$.  It follows that $\fp$ is minimal over $(a)$.  Krull's Principal Ideal
Theorem now gives $\height\fp\le1$; in fact, equality holds, due to the inclusion 
$\fp\supsetneq\fq$.   

Now choose a generator $b$ of $(0:a)_R$.  The exact sequence of $R$-modules
  \[
0\to R/(a)\to R\to R/(b)\to0
  \]
yields $(R/(b))_\fq\cong R_\fq\ne0$, hence $b\in\fq\subseteq\fp$.
Localizing at $\fp$ and changing notation, we get a local ring 
$(R,\fm,k)$ with $\dim R=1$ and elements $a$ and $b$ in $\fm$,
such that the map $R/(a)\to R$ given by $c+(a)\mapsto bc$
is injective.  A result of Fouli and Huneke, see \cite[4.1]{FH}, 
implies that $ab$ is a parameter for $R$; this is absurd, since $ab=0$.

(c)  The ingredients of formula \eqref{eq:dim1} do not change under 
$\fm$-adic completion, so we may assume $R=\wh{\ul R}$ and $I={\ul I}R$, 
with $\ul I$ minimally generated by homogeneous elements 
$\ul a_1,\dots,\ul a_c$ in $\ul R$.  The Koszul complex  $\ul E$ on these
generators is naturally bigraded, so there exist homogeneous cycles 
$\ul z_1,\dots,\ul z_h$  in $\ul E_1$, whose homology classes minimally 
generate the graded $\ul R$-module $\HH1{\ul E}$.  Construction 
\ref{con:tate} then yields  complex $\ul F$ of graded $\ul R$-modules, 
with differentials of degree $0$.

The complex of $R$-modules $R\otimes_{\ul R}\ul F$ is isomorphic to
the Tate complex $F$ on $R\otimes_{\ul R}\ul E$ and $\{1\otimes \ul
z_1,\dots,1\otimes \ul z_h\}$, which is a free resolution of $S$ over $R$
by  \ref{ch:tate}.  On the category of \emph{finite
graded} $\ul R$-modules the functor $(R\otimes_{\ul R}-)$ is exact and
faithful, so  $\ul F$ is a graded free resolution of $\ul S=\ul R/{\ul
I}$ over $\ul R$.  As $(\ul F_n)_j=0$ for $j<n$, we get
  \begin{equation}
    \label{eq:2vars}
\sum_{n=0}^{\infty}\bigg(\sum_{j=n}^\infty\rank_{\ul R}(\ul F_n)_j\,s^j\bigg)t^n
=\frac{\prod_{u=1}^c(1+s^{\deg a_u}t)}{\prod_{v=1}^h(1-s^{\deg z_v}t^2)}\,.
  \end{equation}
by counting ranks of graded free $\ul R$-modules.  Set 
$\Hil Ss=\sum_{j\in\mathbb Z}\rank_kS_j\,s^n\in\mathbb{Z}[[s]]$ and define $\Hil Rs$ similarly.
By counting $k$-ranks in the exact sequence
  \[
\cdots\to \ul F_n\to \ul F_{n-1}\to\cdots\to \ul F_1\to \ul F_0\to \ul
S\to 0
  \]
we obtain the first equality of formal power series in the following formula:
  \[
\Hil{\ul S}s
=\Hil{\ul R}s\cdot\sum_{n=0}^{\infty}(-1)^n\bigg(\sum_{j=n}^\infty
  \rank_{\ul R}(\ul F_n)_j\,s^j\bigg)
=\Hil{\ul R}s\cdot\frac{\prod_{u=1}^c(1-s^{\deg a_u})}{\prod_{v=1}^h(1-s^{\deg z_v})}\,.
  \]
The second equality is obtained by evaluating formula \eqref{eq:2vars} at $t=-1$.

By the Hilbert-Serre Theorem, $\Hil{\ul S}s$ represents a rational
function in~$s$, and the order of its pole at $s=1$ equals the Krull
dimension of $\ul S$.  Equating the orders of the poles at $s=1$ in the
formula above, we get the second equality in the string
  \[
\dim S=\dim\ul S=\dim\ul R+h-c=\dim R+h-c=\dim R-\grade_RS\,.
  \]
The last equality is given by \eqref{eq:grade1}, while the other
two are standard.
 \end{proof}

The theorem above supports the following

  \begin{conjecture}
    \label{con:1}
The equality \eqref{eq:dim1} holds for each quasi-complete intersection~ideal$.$
 \end{conjecture}
 
  \begin{remark}
    \label{grade0}
It suffices to prove Conjecture \ref{con:1} in the special case $\grade_RS=0$. 

Indeed, set $d=\dim R$ and $g=\grade_RS$.  By prime avoidance, $I$ has
a minimal generating set in which the first $g$ elements form an $R$-regular 
sequence, $\bd b$.  By Lemma \ref{lem:step}, the ideal $\ov I=I/(\bd b)$ of
$\ov R=R/(\bd b)$ is quasi-complete intersection.  We now have $S\cong\ov R/\ov I$ 
with $\grade_{\ov R}S=0$, and $\dim\ov R=d-g$.
  \end{remark}

Recall that a finite $R$-module $M$ is called \emph{totally reflexive} if 
$M\cong M^{\vee\vee}$ and $\Ext nRMR=0=\Ext nR{M^\vee}R$ for $n\ge1$;
here $(-)^\vee$ stands for $\Hom R-R$.

  \begin{remark}
A quasi-complete intersection ideal $I$ is quasi-Gorenstein by Theorem 
\ref{thm:hierarchy}(4), so when $\grade_RS=0$ the $R$-module $S$ is
totally reflexive by \ref{ch:gorRing}.  Remark~\ref{grade0} now shows that 
a positive answer to the next question implies Conjecture~\ref{con:1}.
  \end{remark}

  \begin{question}
    \label{question}
Does $\dim_RM=\dim R$ hold for every totally reflexive $R$-module $M$? 
 \end{question}

  \begin{remark}
    \label{CM}
When $R$ is Cohen-Macaulay the answer to Question \ref{question} is 
positive:  If $M$ is totally reflexive, then for every $n\ge0$ it is an $n$th 
syzygy module in some minimal free resolution, so $\dim_RM=\depth R$; 
see, for instance, \cite[1.2.8]{Avr98}.
  \end{remark}

\section{Homology algebras}
  \label{S:Homology algebras}

When $(R,\fm,k)$ is a local ring $\Tor*Rkk$ 
has a structure graded-commutative $k$-algebra.  It is natural in $R$, so 
each ideal $I$ defines a homomorphism  of graded $k$-algebras 
$\Tor*Rkk\to\Tor*{R/I}kk$.  One of the few cases where it is fully understood 
is when $I$ is generated by a regular element; see \cite[3.4.1]{GL}.  

Without using that information, in this section we obtain an explicit result for 
all quasi-complete intersection ideals.  It plays an important role when 
analyzing the transfer of properties between $R$ and $S$; see 
Proposition \ref{prop:ascent/descent}.

   \begin{chunk}
  \label{ch:DGGamma}
A \emph{system of divided powers} on a graded $R$-algebra $A$ is an 
operation that for each $j\ge1$ and each $i\ge0$ assigns to every $a\in
A_{2i}$ an element $a^{(j)}\in A_{2ij}$, subject to certain axioms;
cf.~\cite[1.7.1]{GL}. A DG$\Gamma$ $R$-algebra is a DG $R$-algebra $A$
with divided powers compatible with the differential: $\dd(a^{(j)})=\dd
(a)a^{(j-1)}$.

Let $A$ be a DG$\Gamma$ algebra with $R\to A_0$ surjective and $A_n=0$ for 
$n<0$. Let $A_{\ges1}$ denote the subset of $A$, consisting of elements of
positive degree. Let $D(A)$ denote the graded $R$-submodule of $A$,
generated by $A_0+\fm A_{\ges1}$, all elements of the form  $uv$ with
$u,v\in A_{\ges1}$, and all $w^{(j)}$ with $w\in A_{2i}$, $i\ge1,\
j\ge 2$.  This clearly is a subcomplex of $A$, and it defines a complex
$Q^{\gamma}(A)=A/D(A)$ of $k$-vector spaces.

Given a set $\bd x$ of variables with $|x|\ge 1$ for all $x\in \bd x$ (where $|x|$ 
denotes the degree of $x$), we let $A\langle{\bd x}\rangle$
denote a DG$\Gamma$ algebra with
  \begin{gather*}
A\otimes_R\bw_*^R\bigg(\bigoplus_{\substack{x\in\bd x\\
 |x|\text{ odd}}}Rx\bigg)\otimes_R{\mathsf\Gamma}^R_*\bigg(\bigoplus_{\substack{x\in\bd x\\
 |x|\text{ even}}}Rx\bigg)
  \end{gather*}
as underlying graded algebra and differential compatible with that of $A$
and the divided powers of $x\in\bd x$. For every integer $n$ we set $\bd
x_n=\{x\in\bd x: |x|=n\}$.

A Tate resolution of a surjective ring homomorphism $R\to T$ is a 
quasi-isomor\-phism $R\langle\bd x\rangle\to T$, where $\bd x=\{x_i\}_{i\ges 1}$
and $|x_j|\ge|x_i| \ge1$ holds for all $j\ge i\ge1$.  Such a resolution always exists: 
see \cite[Thm.~1]{Ta}, \cite[1.2.4]{GL}, or \cite[6.1.4]{Avr98}.
  \end{chunk}
 
 \begin{chunk} 
  \label{ch:pi}
Any Tate resolution of $R\to k$ gives $\Tor*Rkk$ a structure of DG$\Gamma$
$k$-algebra, and this structure is independent of the choice of resolution.  We set
  \begin{equation}
    \label{eq:pi}
\pi_*(R)=Q^{\gamma}(\Tor*Rkk)\,.
  \end{equation}
We use the following natural isomorphisms as identifications:
  \begin{equation}
    \label{eq:pi1}
\pi_1(R)\cong\Tor1Rkk\cong\fm/\fm^2\,.
  \end{equation}

The assignment $R\mapsto\pi_*(R)$ is a functor from the category of local
rings and surjective homomorphisms to that of graded $k$-vector spaces.
  \end{chunk}

 \begin{theorem}
\label{thm:comparisonBetti}
Let $(R,\fm,k)$ be a local ring and $I$ an ideal of $R$.  
Set $S=R/I$ and let $\varphi\colon R\to S$ be the natural map.  Let $E$
be the Koszul complex on a minimal generating set of $I$ and
set $H=\HH1E$.

If $I$ is quasi-complete intersection, then there is an exact sequence 
  \begin{equation}
    \label{eq:sixterms}
0\to H/\fm H\to\pi_2(R)\xrightarrow{\pi_2(\varphi)}\pi_2(S)
\xrightarrow{\,\delta\,} I/\fm I\to\pi_1(R)
\xrightarrow{\pi_1(\varphi)}\pi_1(S)\to 0
  \end{equation}
of $k$-vector spaces, and there are isomorphisms of $k$-vector spaces
  \begin{equation}
    \label{eq:pi3}
\pi_n(\varphi)\colon\pi_n(R)\xra{\,\cong\,}\pi_n(S)
  \quad\text{for}\quad n\ge3\,.
  \end{equation}
  \end{theorem}
 
   \begin{Remark}
The statement of the theorem and its proof are reminiscent of 
those of \cite[1.1]{Avr77}, but neither result implies the other one.

Parts of the theorem can also be obtained by using Andr\'e-Quillen homology.

Indeed, $I/\fm I\cong\AQ1{S}Rk$ holds by \cite[6.1]{AnBook}, 
$H/\fm H\cong\AQ2{S}Rk$ by \cite[15.12]{AnBook}, and $\pi_n(R)\cong\AQ n{k}Rk$ 
for $n=1,2$ by \cite[6.1 and 15.8]{AnBook}.  Furthermore, \eqref{eq:jz} applied to the 
homomorphisms $R\to S\to k\xra{=}k$ yields an exact sequence
  \begin{equation}
    \label{eq:pi4}
\cdots\to\AQ{n}kRk\lra\AQ{n}kSk\lra\AQ {n}SRk\xra{\eth_{n}}\AQ{n-1}kRk\to\cdots
  \end{equation}
If $\HH1E$ is free, then $\eth_3$ is injective by \cite[Theorem~1]{RoJPAA}; this yields 
\eqref{eq:sixterms}.

For $n\ge3$ one has $\AQ nSRk=0$ by \ref{ch:aq}.  If $\operatorname{char}(k)=0$, 
then  \cite[19.21]{AnBook} gives $\pi_n(R)\cong\AQ n{k}Rk$ for all $n$, so in this case 
\eqref{eq:pi3} follows from \eqref{eq:pi4}.  However, $\pi_*(R)$ and $\AQ *{k}Rk$
are \emph{not} isomorphic when $\operatorname{char}(k)>0$; see~\cite{AnEM}.
    \end{Remark}

  \begin{construction}
    \label{con:tate-comparison}
Choose a Tate resolution $S\langle{\bd y}\rangle\to k$ with $\dd(S\langle
{\bd y}\rangle)\subseteq\fm(S\langle {\bd y}\rangle)$: this is always
possible, see \cite[1.6.4]{GL} or \cite[6.3.5]{Avr98}.

In view of \ref{ch:tate}, the complex $F$ from Construction \ref{con:tate} gives a Tate
resolution $R\langle{\bd v},{\bd w}\rangle\to S$.  This map can be extended 
to a surjective quasi-isomorphism
  \begin{equation*}
\alpha\col R\langle{\bd v},{\bd w},{\bd{x}} \rangle\to S\langle
\bd{y}\rangle
  \end{equation*}
of DG$\Gamma$ algebras with ${\bd x}=\{x_i\}_{i\ges1}$ and
$\alpha(x_i)=y_i$ for each $i$; see \cite[1.3.5]{GL}.  In particular,
$R\langle {\bd v},{\bd w},{\bd{x}} \rangle$ is a free, but not necessarily
minimal, resolution of $k$ over $R$.
  \end{construction}

 \stepcounter{theorem}
  
\begin{proof}[Proof of Theorem {\rm \ref{thm:comparisonBetti}}]
The homomorphisms of DG$\Gamma$ algebras 
  \[
R\langle{\bd v},{\bd w}\rangle
\hookrightarrow R\langle{\bd v},{\bd w},\bd{x}\rangle
\xrightarrow{\alpha}S\langle\bd{y}\rangle
  \]
induce an exact sequence of complexes of $k$-vector-spaces:
   \begin{equation}
 \label{short}
0\to Q^{\gamma}(R\langle{\bd v},{\bd w}\rangle)
\to Q^{\gamma}(R\langle{\bd v},{\bd w},\bd{x}\rangle)
\xra{Q^{\gamma}(\alpha)} Q^{\gamma}(S\langle{\bd y}\rangle)
\to 0
   \end{equation}
Due to the inclusions $\dd(R\langle{\bd v},{\bd
w}\rangle) \subseteq\fm R\langle{\bd v},{\bd w}\rangle$ and
$\dd(S\langle\bd{y}\rangle)\subseteq\fm S\langle\bd{y}\rangle$,
the complexes at both ends of \eqref{short} have zero differentials.
This gives the following expressions:
  \begin{align}
   \label{eq:Hi1}
\HH n{Q^{\gamma}(R\langle\bd{v},\bd{w}\rangle}
& =\begin{cases}
k\bd v\cong I/\fm I &\text{for }n=1\,,
  \\
k\bd w\cong H/\fm H &\text{for }n=2\,,
  \\
0 &\text{for }n\ne1,2\,.
  \end{cases}
\\
   \label{eq:Hi2}
\HH n{Q^{\gamma}(S\langle\bd{y}\rangle}
&=\ \ \,k\bd y_n \cong\pi_n(S)\hspace{13pt}\text{ for all }n\,.
  \end{align}

The next statement is the key point of the argument.

  \begin{Claim}
For $G=R\langle{\bd v},{\bd w},{\bd x}\rangle$ one has 
$\dd_n(G)\subseteq\fm D_{n-1}(G)$ for $n\ne2$.
  \end{Claim}

Indeed, for $n\ne3$ this follows from the exact sequence \eqref{short} 
and the equalities in \eqref{eq:Hi1} and \eqref{eq:Hi2}.  Next we prove
$\dd(G_3)\subseteq\fm G_2+D_2(G)$.  Write $z\in\dd(G_3)$ as
   \begin{equation}
 \label{eq:cycle}
z=\sum_{x\in{\bd x}_2} a_xx + \sum_{w\in\bd w} b_ww+c\,.
   \end{equation}
with $a_x$ and $b_w$ in $R$ and $c\in D_2(G)$.  One has
$\dd(\alpha(z))=\alpha(\dd(z))=0$, so $\alpha(z)$ is a cycle
the minimal free resolution $S\langle{\bd y}\rangle$; this gives
the inclusion below:
  \[
\sum_{x\in{\bd x}_2} \varphi(a_x)y + \alpha(c)=\alpha(z)\in \fm
S\langle{\bd y}\rangle\,.
  \]
{From} this formula, we conclude that in \eqref{eq:cycle} we have $a_x\in\fm$
for each $x\in{\bd x}_2$.  

Now we show that in \eqref{eq:cycle} each $b_w$ is in $\fm$.  Assume,
by way of contradiction, that $b_w\notin\fm$ holds for some $w\in\bd
w$.  Note that $R\langle{\bd v},{\bd w},{\bd x}_1,{\bd x}_2\rangle$
has an $R$-basis consisting of products involving elements from
${\bd v}\sqcup{\bd x}_1$ and divided powers of elements from ${\bd
w}\sqcup{\bd x}_2$.  When the boundary of an element of $R\langle{\bd
v},{\bd w},{\bd x}_1,{\bd x}_2\rangle$ is written in this basis,
the coefficient of $w^{(j)}$ cannot be invertible; this follows from
the Leibniz rule.  On the other hand, the defining properties of
divided powers, see \cite[1.7.1]{GL}, imply that in the expansion of
$z^{(j)}$ the element $w^{(j)}$ appears with coefficient $(b_w)^j$.
In homology, this means $\cls(z)^{(j)}=\cls(z^{(j)})\ne0$.
Since $\alpha$ is a quasi-isomorphism, for every $j\ge1$ we
get $\cls(\alpha(z))^{(j)}\ne0$ in $\HH{*}{S\langle{\bd y}_1,{\bd
y}_2\rangle}$.  This is impossible, as for $j>\rank_k(\fm/\fm^2)$
the $j$th divided power of every class of even degree in
$\HH{*}{S\langle{\bd y}_1,{\bd y}_2\rangle}$ is equal to zero; see
\cite[1.7]{Avr77}.  This finishes the proof of the claim.

\medskip

The claim gives, in particular, $\dd_n(Q^{\gamma}(R\langle{\bd v},{\bd
w},\bd{x}\rangle))=0$ for all \emph{odd} $n$.  By \cite[3.2.1(iii)
and its proof]{GL}, there is a quasi-isomorphism 
$\rho\colon R\langle{\bd v},{\bd w},\bd{x}\rangle\to R\langle\bd{x}'\rangle$
of DG$\Gamma$ $R$-algebras, such that $R\langle\bd{x}'\rangle$ is
a minimal DG$\Gamma$ algebra resolution of $k$, and the induced
map $Q^{\gamma}(\rho)\colon Q^{\gamma}(R\langle{\bd v},{\bd
w},\bd{x}\rangle)\to Q^{\gamma}(R\langle\bd{x}'\rangle)$ is a
quasi-isomorphism.  Choose, by \cite[1.8.6]{GL}, a quasi-isomorphism
$\sigma\colon R\langle\bd{x}'\rangle\to R\langle{\bd v},{\bd
w},\bd{x}\rangle$ of DG$\Gamma$ $R$-algebras.  The minimality of
$R\langle\bd{x}'\rangle$ implies that the composition $\rho\sigma\colon
R\langle\bd{x}'\rangle\to R\langle\bd{x}'\rangle$ is an isomorphism, see
\cite[1.9.5]{GL}, hence so is the map $Q^{\gamma}(\rho\sigma)$.  It is
equal to $Q^{\gamma}(\rho)Q^{\gamma}(\sigma)$, so $Q^{\gamma}(\sigma)$
is a quasi-isomorphism.  Now form the composition of $k$-linear maps
  \[
 \pi_n(R)\cong
\HH n{Q^{\gamma}(R\langle\bd{x}'\rangle)}\cong \HH
n{Q^{\gamma}(R\langle{\bd v},{\bd w},\bd{x}\rangle)} \xrightarrow{\HH
n{Q^{\gamma}(\alpha)}} \HH n{Q^{\gamma}(S\langle\bd{y}\rangle)}
 \cong\pi_n(S)
  \]
where the first isomorphism is due to the minimality of
$R\langle\bd{x}'\rangle$, the second one is $\HH n{Q^{\gamma}(\sigma)}$,
and the third one is \eqref{eq:Hi2}.  As $\alpha\sigma\colon
R\langle\bd{x}'\rangle\to S\langle\bd{y}\rangle$ induces
the identity on $k$, the composed map is, by definition,
$\pi_n(\varphi)\colon\pi_n(R)\to\pi_n(S)$.  It follows that the homology
exact sequence of the exact sequence \eqref{short} is isomorphic to
  \[
\cdots \to\HH n{Q^{\gamma}(R\langle\bd{v},\bd{w}\rangle}
\to\pi_n(R)\xrightarrow{\pi_n(\varphi)}\pi_n(S)
\xrightarrow{\,\eth_n\,}\HH {n-1}{Q^{\gamma}(R\langle\bd{v},\bd{w}\rangle}
\to\cdots
  \]
In view of the isomorphisms in \eqref{eq:Hi1}, it remains to prove
$\eth_3=0$.  This follows from the construction of the connecting
isomorphism, and the claim with $n=3$.
 \end{proof}

  \section{Poincar\'e series}
    \label{S:Poincare series}

In this section $(R,\fm,k)$ is a local ring, $I$ an ideal, $S=R/I$, and $N$ 
a finite $S$-module. Recall that the {\it Poincar\'e series} of $N$ over $S$ 
is the formal power series
  \[
\Poi NSt=\sum_{n=0}^{\infty}\rank_k\Tor nSkN\,t^n\,.
  \]
Our goal is to relate it to $\Poi NRt$ when $I$ is quasi-complete intersection.
The case $N=k$ is of special interest, as the Poincar\'e series of the residue
field encodes important information on how far the ring is from being regular.  
  \medskip
  
  \setcounter{equation}{0}

The \emph{deviations} of $S$ are defined using the vector spaces in \eqref{eq:pi},
by the formula
  \begin{equation}
    \label{eq:deviation}
\varepsilon_n(S)=\rank_k\pi_n(S) \quad\text{for}\quad n\in\mathbb Z\,.
  \end{equation}
They appear in a well-known formula, see \cite[3.1.3]{GL} or \cite[7.1.3]{Avr98}:
  \begin{equation}
    \label{eq:pdct}
\Poi kSt
=\frac{\prod_{i=0}^{\infty}(1+t^{2i+1})^{\varepsilon_{2i+1}(S)}}
     {\prod_{i=0}^{\infty}(1-t^{2i+2})^{\varepsilon_{2i+2}(S)}}\,.
  \end{equation}

For the next theorems, recall that $\edim R$ stands for $\rank_k(\fm/\fm^2)$.

  \begin{theorem}
    \label{thm:Pchange}
When $I$ is quasi-complete intersection the following equality holds:
  \begin{equation}
    \label{eq:Pchange}
\Poi kSt\cdot\frac{(1-t)^{\edim S}}{(1-t^2)^{\depth S}}
=\Poi kRt\cdot\frac{(1-t)^{\edim R}}{(1-t^2)^{\depth R}}\,.
   \end{equation}
  \end{theorem}

  \begin{proof}
Set $g=\grade_RS$, $h=\rank_S\HH1E$, and $c=\rank_k I/\fm I$.  The
equalities
  \begin{align*}
\Poi kSt
&=\frac{(1+t)^{\varepsilon_1(S)}}
       {(1-t^2)^{c-h+\varepsilon_1(S)-\varepsilon_1(R)+\varepsilon_2(R)}}
       \cdot
       \frac{\prod_{i=1}^{\infty}(1+t^{2i+1})^{\varepsilon_{2i+1}(S)}}
      {\prod_{i=1}^{\infty}(1-t^{2i+2})^{\varepsilon_{2i+2}(S)}}
  \\
&= \frac1{(1-t^2)^g}
     \cdot\frac{(1-t)^{\varepsilon_1(R)}}{(1-t)^{\varepsilon_1(S)}}
     \cdot\frac{(1+t)^{\varepsilon_1(R)}}{(1-t^2)^{\varepsilon_2(R)}}
     \cdot\frac{\prod_{i=1}^{\infty}(1+t^{2i+1})^{\varepsilon_{2i+1}(R)}}
     {\prod_{i=1}^{\infty}(1-t^{2i+2})^{\varepsilon_{2i+2}(R)}}
  \\
&=\frac{(1-t^2)^{\depth S}}{(1-t^2)^{\depth R}}
    \cdot\frac{(1-t)^{\varepsilon_1(R)}}{(1-t)^{\varepsilon_1(S)}}
    \cdot\Poi kRt
   \end{align*}
are obtained by applying \eqref{eq:pdct} and \eqref{eq:sixterms} for the first one, 
\eqref{eq:grade1} and \eqref{eq:pi3} for the second, \eqref{eq:cmd2} and \eqref{eq:pdct} for the third.  
Finally, $\varepsilon_1(S)=\edim S$ by \eqref{eq:pi1}.
  \end{proof}

\begin{theorem}
\label{PM-comparison}
If $I$ is a quasi-complete intersection ideal satisfying $I\cap\fm^2\subseteq I\fm$, 
then for every finite $S$-module $N$ the following equality holds:
  \begin{equation}
    \label{eq:PM-comparison}
\Poi NSt\cdot\frac{(1-t)^{\edim S}}{(1-t^2)^{\depth S}}
=\Poi NRt\cdot\frac{(1-t)^{\edim R}}{(1-t^2)^{\depth R}}\,.
   \end{equation}
\end{theorem}

  \begin{proof}
As $I$ is quasi-complete intersection, the map $\pi_n(\varphi)\colon\pi_n(R)\to
\pi_n(S)$ is surjective for $n\ne2$, by Theorem \ref{thm:comparisonBetti}.  The 
hypothesis $I\cap\fm^2\subseteq\fm I$ implies that $I/\fm I\to \fm/\fm^2$ is
injective, so the same theorem shows that $\pi_2(\varphi)$ is
surjective as well.  By the definition of $\pi_*(R)$, the image of any $k$-linear
right inverse $\sigma\colon \pi_*(R)\to\Tor*Rkk$ of the natural surjection
$\Tor*Rkk\to\pi_*(R)$ generates $\Tor*Rkk$ as a graded $\Gamma$-algebra
over $k$.  Thus, the surjectivity of $\pi_*(\varphi)$
means that the map of $\Gamma$-algebras $\Tor*{\varphi}kk\colon 
\Tor*Rkk\to\Tor*Skk$ is surjective; that is, $\varphi$ is a large homomorphism.  
A theorem of Levin, \cite[1.1]{Lev80}, then gives
  \[
{\Poi NSt}{\Poi kRt}={\Poi NRt}{\Poi kSt}\,.
   \]

Now replace $\Poi kSt$ with its expression from Theorem
\ref{thm:Pchange}, and simplify.
  \end{proof}

To finish, we compare our results with earlier ones for complete intersections.  

  \begin{remark}
    \label{rem:Snamash}
Let $\bd a$ be a regular sequence and $I=(\bd a)$.  

Formula \eqref{eq:Pchange} then specializes to theorems of Tate \cite[Theorem\, 4]{Ta} and Scheja 
\cite[Satz\, 1]{Sc}, and \eqref{eq:PM-comparison} to one of Nagata \cite[27.3]{Na}.  It suffices to prove 
those theorems for principal ideals, but such a reduction is impossible here; see Remark~\ref{counterex}. 

When $\bd a$ lies in $\fm\Ann_RN$ one has $\edim S=\edim R$, so \eqref{eq:PM-comparison} holds by
a result of Shamash, \cite[\S 3, Corollary (2)]{Sh}; see also \cite[3.3.5(2)]{Avr98}.  
We know of no analog of that result for quasi-complete intersection ideals.
  \end{remark}

In Theorem \ref{PM-comparison} the hypothesis $I\cap\fm^2\subseteq I\fm$ is essential:

  \begin{example}
    \label{nonShamash}
For $R=k[\![x]\!]$, $S=R/(x^2)$, and $N=S$ one has 
  \[
\Poi NSt\cdot\frac{(1-t)^{\edim S}}{(1-t^2)^{\depth S}}=1-t\ne1
=\Poi NRt\cdot\frac{(1-t)^{\edim R}}{(1-t^2)^{\depth R}}\,.
   \]
  \end{example}

\section{Local homomorphisms of local rings}
  \label{S:Local homomorphisms of local rings}

In this section $\vf\colon(R,\fm,k)\to(S,\fn,l)$ denotes a homomorphism of local 
rings, which is \emph{local}, in the sense that it satisfies $\vf(\fm)\subseteq\fn$. 

We define and study quasi-complete intersection local homomorphisms, with an 
emphasis on the transfer of local ring-theoretic properties between $R$ and $S$.

  \medskip

By \cite[1.1]{AFH}, there is a commutative diagram of local homomorphisms
  \begin{equation}
    \label{eq:cohen}
  \begin{gathered}
\xymatrixrowsep{2pc}
\xymatrixcolsep{2.5pc}
\xymatrix{
&R'
\ar@{->>}[dr]^-{\vf'}
  \\
R
\ar@{->}[r]^-{\vf} 
\ar@{>->}[ur]^-{\dot\vf}
&{S}
\ar@{->}[r]^-{\sigma{\hphantom{}}}
&{\wh{S}}
 }
 \end{gathered}
  \end{equation}
where $\dot\vf$ is flat, $\vf'$ is surjective, $\sigma$ is the $\fn S$-adic completion map, 
$R'$ is complete, and $R'/\fm R'$ is regular.  Any such a diagram is called 
a \emph{Cohen factorization} of $\vf$.

  \begin{chunk}
    \label{ch:cohen}
We say that $\vf\col R\to S$ is \emph{quasi-complete intersection} (or 
\emph{q.c.i.}) \emph{at} $\fn$ if in some Cohen factorization of $\vf$ the ideal 
$\Ker(\vf')$ is quasi-complete intersection.  We first show that this 
property does not depend on the choice of Cohen factorization:
    \end{chunk}

 \begin{lemma}
      \label{lem:independence}
The homomorphism $\vf$ is q.c.i.\ at $\fn$ (if and) only if 
$\Ker\vf''$ is quasi-complete intersection for every 
Cohen factorization $R\xra{\ddot\vf}R''\xra{\vf''}\wh S$
of $\vf$.
  \end{lemma}

 \begin{proof}  
By \cite[1.2]{AFH}, there is a commutative diagram of local homomorphisms
   \[
\xymatrixrowsep{2pc}
\xymatrixcolsep{2.5pc}
\xymatrix{
&
R'
\ar@{->>}[dr]^{\vf'}
  \\
R
\ar@{>->}[r] 
\ar@{>->}[ur]^{\dot\vf}
\ar@{>->}[dr]_{\ddot\vf}
&{\widetilde R}
\ar@{->>}[r]
\ar@{->>}[u]
\ar@{->>}[d]
&
{\wh S}
  \\
&R''
\ar@{->>}[ur]_{\vf''}
 }
  \]
where the middle row is a Cohen factorization of $\vf$, and the vertical 
arrows are surjections with kernels generated by regular sequences.  
Lemma \ref{lem:step} shows that $\Ker(\vf'')$ is quasi-complete intersection 
if and only if $\Ker(\vf')$ is.
 \end{proof}  

 \begin{remark}
      \label{rem:independence}
When $\vf$ is surjective, it is q.c.i.\ at $\fn$ if and only the ideal $I=\Ker(\vf)$ is quasi-complete intersection.

Indeed, $R\to\wh R\xra{\wh\vf}\wh S$ clearly is a Cohen factorization, so by the preceding 
lemma it suffices to show that $I$ is quasi-complete intersection if and only if $\Ker(\wh\vf)$ is one.   
In view of the equality $\Ker(\wh\vf)=I\wh R$, this follows from Lemma \ref{lem:qciflat}.
  \end{remark}

The homomorphism $\vf$ is said to be \emph{complete intersection} (or \emph{c.i.}), respectively, 
\emph{quasi-Gorenstein} at $\fn$ if in some Cohen factorization the ideal $\Ker(\vf')$ has the property 
described in \ref{ex:ci}, respectively, in \ref{ch:gorRing}.  As above, this notion does not depend on the 
choice of factorization; see \cite[(3.3)]{Avr99} and \cite[4.3]{AF:qGor}, respectively.

 \begin{proposition}
  \label{prop:hierarchy}
The homomorphism $\vf$ is c.i.\ at $\fn$ if and only if it is q.c.i.\ at $\fn$ and  
$S$ has finite flat dimension over $R$.

If $\vf$ is q.c.i.\ at $\fn$, then it is quasi-Gorenstein at $\fn$.
  \end{proposition}

\begin{proof}  
Choose any Cohen factorization \eqref{eq:cohen} and set $I=\Ker(\vf')$.

By \cite[3.3]{AFH}, $S$ has finite flat dimension over $R$ if and only if $I$ 
has finite projective dimension over $R'$, so the first assertion follows from Theorem \ref{thm:hierarchy}(1).  
 
If $\vf$ is q.c.i.\ at $\fn$, then the ideal $I$ is quasi-complete intersection by 
Lemma \ref{lem:independence}, and hence it is quasi-Gorenstein by Theorem \ref{thm:hierarchy}(4).  
 \end{proof}

Next we relate certain local properties of $R$ and $S$.  If $\vf$ is 
quasi-Gorenstein at $\fn$, then the rings $R$ and $S$ are 
simultaneously Gorenstein by \cite[7.7.2]{AF:qGor}.  In view of \ref{hierarchy},
the conclusion holds when $\vf$ is q.c.i.\ at $\fn$; for surjective $\vf$ 
this is already noted in \cite[Cor.\,5]{GS}, along with the fact that $S$ 
is Cohen-Macaulay when $R$ is.

We want to compare numerical measures of the singularity of local rings. 

The \emph{Cohen-Macaulay defect} of $S$ is the non-negative number
  \begin{equation*}
\cmd S=\dim S-\depth S\,;
  \end{equation*}
it is equal to zero if and only if $S$ is Cohen-Macaulay.  

Similarly, the 
\emph{complete intersection defect} of $S$ is the non-negative number
  \begin{equation*}
\cid S =\varepsilon_2(S)-\varepsilon_1(S)+\dim S\,;
 \end{equation*}
it is equal to zero if and only if $S$ is complete intersection; see \cite[2.3.3(b)]{BH}.

 \begin{proposition} 
    \label{prop:ascent/descent}
When $\vf$ is q.c.i.\ at $\fn$ the following inequalities are satisfied:
  \[
\cmd S\le\cmd R
\quad\text{and}\quad \cid S\le\cid R\,.
  \]
Equalities hold if $\grade_RS=\dim R-\dim S$; in particular, if $R$ is Cohen-Macaulay. 
 \end{proposition}

\begin{Remark}
Conjecture~\ref{con:1} predicts that equalities always hold: See \eqref{eq:ascent/descent1}
and \eqref{eq:ascent/descent2}.
  \end{Remark}

\begin{proof}  
Let $R\xra{\dot\vf}R'\xra{\vf'}\wh{S}$ be a Cohen factorization.
The functions $\cmd$ and $\cid$ are additive on flat extensions by
\cite[1.2.6, A.11]{BH} and \cite[3.6]{Avr77}, respectively.  As they 
vanish on the regular rings $R'/\fm R'$ and $\wh{S}/\fn\wh{S}$, 
we see that $R$ and $R'$ are Cohen-Macaulay simultaneously,
and that we may assume the map $\vf$ is surjective.  

From the definition and formula \eqref{eq:cmd2} we now obtain an equality
  \begin{equation}
    \label{eq:ascent/descent1}
\cmd S=\cmd R-(\dim R-\dim S -\grade_RS)\,.
  \end{equation}

On the other hand, the definition, \eqref{eq:sixterms}, and \eqref{eq:grade1} yield
   \begin{equation}
    \label{eq:ascent/descent2}
\begin{aligned}
\cid S
&=\varepsilon_2(S)-\varepsilon_1(S)+\dim S\\
&=\varepsilon_2(R)-\varepsilon_1(R)+\rank_k(I/\fm I)-\rank_S\HH1E+\dim S\\
&=\cid R-(\dim R-\dim S -\grade_RS)\,. 
   \end{aligned}
  \end{equation}
  
The desired assertions follow because $\dim R\ge\dim S+\grade_RS$ always holds, 
see \cite[p.\, 413 and 1.2.14]{BH}, with equality if $R$ is Cohen-Macaulay, see \cite[2.1.4]{BH}.
 \end{proof}

For surjective maps $\vf$ the next corollary is proved in \cite[Corollary 5]{GS}.

\begin{corollary} 
    \label{cor:ascent/descent}
Assume that $\vf$ is q.c.i.\ at $\fn$.

If the ring $R$ is Cohen-Macaulay, then so is $S$.

The ring $R$ is Gorenstein if and only if so is $S$.
 \end{corollary}

\begin{proof}  
The first assertion is already contained in the theorem.  The second one comes from \cite[7.7.2]{AF:qGor}, 
since $\vf$ is quasi-Gorenstein at $\fn$ by Proposition \ref{prop:hierarchy}.
 \end{proof}

The next result can also be obtained from \cite[1.5]{Avr99} or \ref{andre} below. 

\begin{proposition}
    \label{thm:andre}
Any two conditions below imply the third one$.$
  \begin{enumerate}[\quad\rm(a)]
 \item
The homomorphism $\vf$ is q.c.i.\ at $\fn$.
 \item
The ring $R$ is complete intersection.
  \item
The ring $S$ is complete intersection.
   \end{enumerate}
 \end{proposition}

\begin{proof}  
As in the proof of Proposition \ref{prop:ascent/descent}, we may assume that $\vf$ is surjective. 

When (a) holds, Theorem \ref{thm:comparisonBetti} gives 
$\varepsilon_3(R)=\varepsilon_3(S)$.  Vanishing of $\varepsilon_3$ 
characterizes complete intersections, see \cite[3.5.1(iii)]{GL} or 
\cite[7.3.3]{Avr98}, whence (b)$\iff$(c).

When (b) and (c) hold, and $\varkappa\col Q\to\wh R$ is a surjective homomorphism 
with $Q$ regular local, then both ideals $I=\Ker\varkappa$ and 
$J=\Ker(\wh\vf \varkappa)$ are generated by regular sequences.  By \ref{ex:ci} and 
Lemma \ref{lem:step}, $J\wh R$ is quasi-complete intersection.  Now $J\wh R=I\wh R$,
so $I$ is quasi-complete intersection by Lemma \ref{lem:qciflat}.
 \end{proof}

\section{Homomorphisms of noetherian rings}
  \label{S:Homomorphisms of noetherian rings}

In this section $\vf\colon R\to S$ denotes a homomorphism of noetherian rings.
 
For such homomorphisms we first define the q.c.i.\ property in terms of localizations,
then show show how they can be described in terms of vanishing of Andr\'e-Quillen 
homology.  Various properties of that theory, collected in Appendix~\ref{S:Andre-Quillen 
homology}, are heavily used when studying the stability of this class of maps.

Developments here follow the treatment of l.c.i.\ homomorphisms 
in~\cite[\S 5]{Avr99} and proofs that can be transposed with only 
superficial changes have been omitted.  
  \medskip
  
For $\fq\in\Spec S$ we let $\fq\cap R$ denote the prime ideal $\vf^{-1}(\fq)$ of~$R$.  
As usual, we set $k(\fq)=S_\fq/\fq S_\fq$ and call $k(\fq)\otimes_RS$ the 
\emph{fiber} of $\vf$ at $\fq$.  The induced homomorphism of local rings 
$\vf_\fq\col R_{\fq\cap R}\to S_\fq$ is called the \emph{localization} of $\vf$ at $\fq$.  
  
  \begin{chunk}
We say that the homomorphism $\vf$ is \emph{quasi-complete intersection} 
(or \emph{q.c.i.}) if it is q.c.i.\ at every $\fq\in\Spec S$.  
  \end{chunk}

This notion mimics those of \emph{locally complete intersection} (or \emph{l.c.i}) 
homomorphism in \cite{Avr99}  and of \emph{quasi-Gorenstein} 
homomorphism in \cite{AF:qGor}, defined by the corresponding condition on $\vf_\fq$.  
Proposition \ref{prop:hierarchy} clarifies the relationships:

 \begin{bchunk}{\rm\textbf{Hierarchy.}}
  \label{hierarchy}
The homomorphism $\vf$ is l.c.i.\ if and only if it is q.c.i.\ and the $R$-module $S_\fq$ has 
finite flat dimension for every prime ideal $\fq$ of $S$.

If $\vf$ is q.c.i., then it is quasi-Gorenstein.
 \qed
  \end{bchunk}

The concept fits properly into Grothendieck's theory of flat maps; see \cite[\S 6]{Gr}.

 \begin{bchunk}{\rm\textbf{Flat homomorphisms.}}
  \label{flat}
When $\vf$ is flat it is q.c.i.\ if and only if it is l.c.i., if and only if the ring $S\otimes_Rk(\fq)$ 
is locally complete intersection for every $\fq\in\Supp S$.
  \end{bchunk}

 \begin{proof}
The first equivalence holds because for flat homomorphisms the q.c.i.\ and l.c.i.\ conditions 
coincide, by \ref{hierarchy}.  The second equivalence is \cite[5.2]{Avr99}.
  \end{proof}

\stepcounter{theorem}

Next we describe q.c.i.\ homomorphisms in terms of Andr\'e-Quillen homology, by extending
from the case of surjective maps the characterization given in \cite{BMR}.

\begin{bchunk}{\rm\textbf{Andr\'e-Quillen homology.}}
     \label{prop:aq}
The homomorphism $\vf$ is q.c.i.\ at $\fq\in\Spec S$ if and only if $\AQ nSR{k(\fq)}=0$ 
holds for $n\ge3$.

The homomorphism $\vf$ is q.c.i.\ if and only if $\AQ nSR-=0$ holds for $n\ge3$.
  \end{bchunk}

 \begin{proof}
Let $R_{\fq\cap R}\to R'\to\wh{S_\fq}$ be a Cohen factorisation.  The kernel of $R'\to\wh{S_\fq}$ 
is quasi complete intersection if and only if $\AQ n{\wh{S_\fq}}{R'}{k(\fq)}=0$ for $n\ge3$; see \ref{ch:aq}.  
By \ref{eq:iso2}, this is equivalent to $\AQ nSR{k(\fq)}=0$ for $n\ge3$, whence the first assertion.

The second assertion follows from the first one by \cite[S.29]{AnBook}.
  \end{proof}

The next result can be viewed as a sharper version of Proposition \ref{thm:andre}.
For surjective $\vf$ it is due to Andr\'e \cite[Theorem and Proposition]{AnJPAA}.  

\begin{bchunk}{\rm\textbf{Complete intersections.}}
    \label{andre}
Any two conditions below imply the third one.
  \begin{enumerate}[\quad\rm(a)]
 \item
$\AQ 3SR{k(\fq)}=0$.
 \item
The ring $R_{\fq\cap R}$ is complete intersection.
  \item
The ring $S_\fq$ is complete intersection.
   \end{enumerate}
When these conditions hold the homomorphism $\vf$ is q.c.i.\ at $\fq$.
 \end{bchunk}

\begin{proof}
By \cite[4.5]{Avr99}, $S_\fq$ is complete intersection if and only it $\mathbb{Z}\to S_\fq$ is 
c.i.~at~$\fq S_\fq$.  The exact sequence \eqref{eq:jz} for $\mathbb{Z}\to R_{\fq\cap R}\to S_{\fq}\to k(\fq)$ 
and the criteria for c.i.\ and for q.c.i.\ homomorphisms in \ref{lci} and \ref{prop:aq}, respectively,
yield the assertions.
  \end{proof}

The proof of \cite[5.11]{Avr99} shows that the following result is a consequence of  
\ref{prop:aq} and standard properties of Andr\'e-Quillen homology with respect
to flat base change.

\begin{bchunk}{\rm\textbf{Flat base change.}}
  \label{flatbasechange}
Let $R'$ be a noetherian ring, $\rho\col R\to R'$ a homomorphism of rings such that 
$S\otimes_RR'$ is noetherian, and set $\vf'=\vf\otimes_RR'\col R'\to S\otimes_RR'$.
\begin{enumerate}[\rm(1)]
\item
If $\vf$ is q.c.i.\ and $\Tor nR{S}{R'}=0$ holds for all $n\ge1$, then $\vf'$ is q.c.i.
\item
If $\vf'$ is q.c.i.\ and $\rho$ is faithfully flat, then $\vf$ is q.c.i.
    \qed
  \end{enumerate}
   \end{bchunk}

The next three items involve also a noetherian ring $Q$ and a homomorphism of rings 
$\psi\col Q\to R$.  For those assertions that come in two versions it clearly suffices to 
prove the statement  that includes the text in parentheses.
 
\begin{bchunk}{\rm\textbf{Composition.}}
  \label{composition}
If $\vf$ is q.c.i.\ (at some $\fq\in\Spec S$) and $\psi$ is q.c.i.\ (at $\fq\cap R$), then 
the composed homomorphism $\vf\psi$ is q.c.i.\ (at $\fq$).
  \end{bchunk}

 \begin{proof}
The exact sequence \eqref{eq:jz} for $Q\to R\to S\to k(\fq)$ yields $\AQ{n}SQ{k(\fq)}=0$ 
for $n\ge3$, due to \ref{prop:aq}.  By the same token, $\vf\psi$ is q.c.i.\ at $\fq$.
  \end{proof}
   
\begin{bchunk}{\rm\textbf{Decomposition.}}
  \label{decomposition}
Assume that $\vf\psi$ is q.c.i.\ (at some $\fq\in\Spec S$).
\begin{enumerate}[\rm(1)]
\item
If $\psi$ is l.c.i.\ (at $\fq\cap R$), then $\vf$ is q.c.i.\ (at $\fq$).
\item
If $\vf$ is q.c.i.\ (at $\fq$), then $\psi$ is q.c.i.\ (at $\fq\cap R$).
\item
If $\fd_RS_\fq$ is finite for some $\fq\in\Spec S$, then $\vf$ is c.i.\ at $\fq$
and $\psi$ is q.c.i.\ at~$\fq\cap R$.
\end{enumerate}
\end{bchunk}

 \begin{proof}
(1) Set $l=k(\fq)$.  The exact sequence \eqref{eq:jz} for $Q\to R\to S\to l$ yields $\AQ{n}SRl=0$ 
for $n\ge3$ by \ref{prop:aq} and \ref{lci}, so $\vf$ is q.c.i.\ at $\fq$ by \ref{prop:aq}.

(2) The same exact sequence as in (1) here yields $\AQ{n}RQ{k(\fq\cap R)}=0$ for 
$n\ge3$ by \ref{prop:aq}, and the latter also shows that $\psi$ is q.c.i.\ at $\fq\cap R$.

(3) Since $\AQ{n}SQl\to\AQ{n}SRl$ is surjective for $n=4$ if $\cha l\ne2$ and for 
$n=3$ if $\cha l=2$, see \ref{thm:eth}, we get $\AQ{4}SRl=0$ if $\cha l\ne2$ and 
$\AQ{3}SRl=0$ if $\cha l=2$ from \ref{prop:aq}. Now $\vf$ is c.i.\ at~$\fq$ by \ref{lci}, 
so $\psi$ is q.c.i.\ at $\fq\cap R$ by~(2). 
  \end{proof}

\begin{bchunk}{\rm\textbf{Flat descent.}}
  \label{descent}
When $\vf$ is faithfully flat the composed homomorphism $\vf\psi$ is q.c.i.\ if and only if 
$\vf$ is l.c.i.\ and $\psi$ is q.c.i.
\end{bchunk}

 \begin{proof}
The ``if'' part comes from \ref{composition}.  The converse follows from \ref{decomposition}(3), 
because faithfully flat homomorphisms induce surjections on spectra.
   \end{proof}

By \ref{flat}, the following result applies to homomorphisms essentially of finite type.

\begin{bchunk}{\rm\textbf{Factorizable homomorphisms.}}
  \label{factorization}
Assume $\vf=\vf'\dot\vf$, where $\dot\vf$ and $\vf'$ are homomorphism of rings
such that $\dot\vf$ is l.c.i.\ and $\vf'$ is surjective.

The homomorphism $\vf$ is q.c.i.\ if and only if $\Ker(\vf')_{\fm'}$ is a quasi-complete 
intersection ideal of $R'_{\fm'}$ for every maximal ideal $\fm'$ of $R'$ containing $\Ker(\vf')$.
   \end{bchunk}

  \begin{proof}
From \ref{composition} and \ref{decomposition}(1) we see that $\vf$ is q.c.i.\ if and only if $\vf'$ is.  
By Remark \ref{rem:independence} the latter holds if and only if the ideal $\Ker(\vf')_{\fp'}$ ot $R'_{\fp'}$ 
is quasi-complete intersection for every prime ideal $\fp'$ of $R'$ that contains $\Ker(\vf')$.  
Lemma \ref{lem:qciflat} implies that this can be verified by checking only those $\fp'$ that are maximal.
  \end{proof}

For homomorphisms covered by \ref{factorization} the q.c.i.\ property \emph{localizes},
in the sense that if it holds at the maximal ideals of $S$, then it does at all prime ideals. 
In general, localization may fail. Indeed, recall that the non-zero fiber rings of the completion 
maps $R\to\wh{R_\fm}$, when $\fm$ ranges over the maximal ideals of $R$, are called 
the \emph{formal fibers} of the ring $R$.    Ferrand and Raynaud \cite[3.2(i)]{FR} 
give a local domain $R$ with $\wh R\otimes k(0)$ not Gorenstein; thus, 
$R\to\wh R$ is q.c.i.\ at $\fm$ but not q.c.i., see~\ref{flat}.

In our second result concerning localization the hypotheses are placed on $R$, rather than on~$\vf$.  
It applies, for instance, to excellent rings (as their formal fibers are regular) and to homomorphic 
images of l.c.i.\ rings (see \cite[Main Theorem (b)]{AF:Gro}).

\begin{bchunk}{\rm\textbf{Localization.}}
  \label{localization}
Assume that $R_{\fq\cap R}$ has l.c.i.\ formal fibers for each $\fq$ in $\Spec S$.

\hskip-1mm When $\vf$ is q.c.i.\,at the maximal ideals of $S$,\,it is q.c.i.\,and $S$\! has 
l.c.i.\,formal fibers{.}
  \end{bchunk}

  \begin{proof}
Choose $\fq\in\Spec S$, a maximal ideal $\fn$ of $S$ with $\fn\supseteq\fq$, 
and set $\fm=\fn\cap R$.  Set $R^*=\wh{R_\fm}$ and $S^*=\wh{S_\fn}$, and let 
$R^*\to R'\to S^*$ be a Cohen factorization of $\vf^*=\wh{\vf_{\fn}}\col R^*\to S^*$.  
Choose $\fq^*\in\Spec S^*$ with  $\fq^*\cap S=\fq$, and set $\fp'=\fq^*\cap R'$.  

The map $R^*\to R'$ is c.i.\ at $\fp'$ as it is flat, $R^*$ and $R'$ are complete, 
and $R'/\fm R'$ is regular; see \cite[\S3, Step~1]{AF:Gro}.  By \ref{lci} this gives
$\AQ n{R'}{R^*}{k(\fp')}=0$ for $n\ge2$.  

In addition, $R_\fm\to R'\to S^*$ is a Cohen factorization of  $\vf_\fn$, so $\vf^*$ is q.c.i.\ 
at $\fn S^*$ by Lemma \ref{lem:independence},  hence $\AQ n{S^*}{R'}{k(\fq^*)}=0$ for 
$n\ge3$ by \ref{prop:aq}.  The exact sequence \eqref{eq:jz} for $R^*\to R'\to S^*\to k(\fq^*)$ 
now yields $\AQ n{S^*}{R^*}{k(\fq^*)}=0$ for $n\ge3$.

As $R_\fm$ has l.c.i.\ formal fibers, \ref{lci} gives $\AQ n{R^*}{R}{k(\fq^*)}$ 
for $n\ge2$.  The sequence \eqref{eq:jz} defined by
$R\to R^*\to S^*\to k(\fq^*)$ gives $\AQ n{S^*}{R}{k(\fq^*)}=0$ 
for $n\ge3$, so the composed map $R\to S\to S^*$ is q.c.i.\ at  $\fq^*$ 
by \ref{prop:aq}. Since $S\to S^*$ is flat, \ref{decomposition}(3) shows 
that it is c.i.\ at $\fq^*$ and $R\to S$ is q.c.i.\ at~$\fq$.  
 
We obtain the desired assertions by varying the choices of $\fq$, $\fn$, and $\fq^*$.
  \end{proof}

The proof of \cite[6.11]{AF:Gor} shows that \ref{localization} and \ref{flatbasechange} 
imply the next property.

\begin{bchunk}{\rm\textbf{Completion.}}
Assume that $R_{\fq\cap R}$ has l.c.i.\ formal fibers for each $\fq\in\Spec S$.

Let ${I} \subseteq R$ and ${J} \subsetneq S$ be ideals, such that
$\vf({I})\subseteq{J}$, and let $\vf^*\col R^*\to S^*$ be the induced
map of the corresponding ideal-adic completions. 

\begin{enumerate}[\rm(1)]
\item
If $\vf$ is q.c.i., then so is $\vf^*$.
\item
If $I$ is contained in the Jacobson radical of $R$ and $\vf^*$ is q.c.i., then so is $\vf$.
  \qed
  \end{enumerate}
\end{bchunk}

\appendix

\section{Andr\'e-Quillen homology}
  \label{S:Andre-Quillen homology}

For each $A$-algebra $B$ and every $B$-module $N$, Andr\'e \cite{AnBook} and Quillen \cite{Qu} 
constructed homology groups $\AQ nBAN$ with remarkable functorial properties.  A few results 
crucial for this paper, taken from \cite{AnBook} and \cite{Avr99}, are collected below. 

Let $A\xra{\alpha}B\xra{\beta}C\xra{\gamma}l$ be homomorphisms of noetherian rings,
where $l$ is a field.

  \begin{chunk}
    \label{jz}
For $k=k(\Ker\gamma)$ and each $n\ge0$ there is a natural exact sequence 
of $l$-modules
  \begin{equation}
    \label{eq:jz}
\AQ{n}BAk\otimes_kl\lra\AQ{n}CAl\lra\AQ {n}CBl\xra{\eth_{n}}\AQ{n-1}BAk\otimes_kl\,.
  \end{equation}

Indeed, $\alpha$ and $\beta$ define a Jacobi-Zariski exact sequence with coefficients in $l$, see \cite[5.1]{AnBook}, 
which differs from \ref{eq:jz} only because in it $\AQ{n}BAl$ appears in place of $\AQ{n}BAk\otimes_kl$.  However, 
these modules are isomorphic by \cite[4.58]{AnBook}.
   \end{chunk}

For $\fq\in\Spec C$ let $\wh{C_\fq}$ denote the $\fq C_\fq$-adic completion of $C_\fq$ 
and set $\fp=\fq\cap B$.

    \begin{chunk}
     \label{lem:aq-local}
By \cite[4.58 and 5.27]{AnBook} there are natural isomorphisms
  \begin{equation}
    \label{eq:iso1}
\AQ nCB{k(\fq)}
\cong\AQ n{C_\fq}B{k(\fq)}
\cong\AQ n{C_\fq}{B_{\fp}}{k(\fq)}
  \quad\text{for}\quad n\ge0
  \end{equation}
If $B_{\fp}\to B'\to\wh{C_\fq}$ is a Cohen factorization of $\beta_\fq$,  then
\eqref{eq:iso1} and \cite[1.7]{Avr77} give
  \begin{equation}
    \label{eq:iso2}
\AQ nCB{k(\fq)}\cong
\AQ n{\wh{C_\fq}}{B'}{k(\fq)} 
  \quad\text{for}\quad n\ge2\,.
  \end{equation}
    \end{chunk}

  \begin{chunk}
    \label{thm:eth}
Set $l=k(\fq)$ and let $\gamma\col C\to l$ be the canonical surjection.

If $\fd_BC_\fq$ is finite, then in \eqref{eq:jz} one has $\eth_n=0$ in the following cases:  
\begin{enumerate}[\quad\rm(a)]
\item
$n=2i$ for some integer $i$ with $1\le i<\infty$ and $\cha l=0$.
\item
$n=2i$ for some integer $i$ with $1\le i\le\cha l$.
\item
$n=3$ and $\cha l=2$.
  \end{enumerate}
  
Indeed, in view of \eqref{eq:iso1} we may assume $B=B_\fp$, $C=C_\fq$, and the homomorphisms 
$\alpha$ and $\beta$ are local.  Cases (a) and (b) then are settled by \cite[4.7]{Avr99}.  
When $\cha l=2$ the map $\pi_n(\alpha)\col\pi_n(A)\otimes_kl\to\pi_n(B)$ is injective 
for {all} integers $n\ge2$, not just for the even ones, see~\cite{Avr82}, so case (c) follows from the 
\emph{proof} of \cite[4.7]{Avr99}
 \end{chunk}
 
   \begin{chunk}
     \label{lci}
For $\fq\in\Spec C$  the following conditions are equivalent.
\begin{enumerate}[\quad\rm(i)]
\item
$\beta$ is c.i.\ at $\fq$.
\item
$\AQ {\ges 2}CB{k(\fq)}=0$.
\item
$\AQ 2CB{k(\fq)}=0$.
\item
$\AQ {\ges q}CB{k(\fq)}=0$ for some some integer $q$, and $\fd_BC_\fq$ is finite.
\item
$\AQ nCB{k(\fq)}=0$ for some integer $n$ with $3\le n\le 2m-1$, where $m$ is an integer 
such that $(m-1)!$ is invertible in $k(\fq)$, and $\fd_BC_\fq$ is finite.
\end{enumerate}

This is a consequence of \cite[1.2, 1.3, and 1.4]{Avr99}, via \eqref{eq:iso1}.
 \end{chunk}

We finish with two open questions suggested by results in the main text.
 
  \begin{chunk} 
     \label{open}
The equivalence of conditions (i) and (ii) in \ref{lci} has a parallel in \ref{prop:aq}.  
It is natural to ask whether analogs of other conditions hold:
   \begin{enumerate}[\quad\rm(i)]
\item[\rm(iii)]
Does $\AQ 3CB{k(\fq)}=0$ imply that $\beta$ is q.c.i.\ at $\fq$?  
\item[\rm(iv)]
Does $\AQ {\ges q}CB{k(\fq)}=0$ for some some integer $q$ imply that $\beta$ is q.c.i.~at~$\fq$? 
   \end{enumerate}

Special cases of (iii) are covered by \ref{andre}.  An affirmative answer to (iv) was conjectured 
by Quillen \cite[5.6]{Qu} when $\beta$ is of finite type and in \cite[p.\,459]{Avr99} in general. 
That conjecture was proved in \cite[1.3]{Avr99} in case $C_\fq$ has finite flat dimension over 
$B$, and in \cite[Theorem\,1]{AI0} in case $\beta$ admits a right inverse ring homomorphism.
  \end{chunk}
  
\section*{Acknowledgement}

We thank Javier Majadas Soto for pointing us to his work on the subject, 
and Ray Heitmann, Craig Huneke, and Srikanth Iyengar for useful conversations.

  \end{document}